\documentclass[4pt, oneside]{article}
\usepackage{amsmath,amssymb,amsthm,units, stmaryrd,stackrel,fouridx,relsize}
\usepackage{empheq,bm,url}
\usepackage{yfonts}
\usepackage{units,stmaryrd}%not jsl
\usepackage{color}
\usepackage{graphicx}
\usepackage[all]{xy}
\newtheorem{theorem}{Theorem}[section]

\newtheorem{definition}[theorem]{Definition}
\newtheorem{lemma}[theorem]{Lemma}

\newtheorem{corollary}[theorem]{Corollary}

\newcommand\CurrentProof[1]{#1}
\newcommand\AlternativeProof[1]{}

\newcommand{\snapshot}[2]{{#2}:\mathfrak W
\mathrel{\vcenter{\offinterlineskip
\hbox{\hskip.6ex${\mathsmaller{#1}}$}\vskip 0.2ex \hbox{$\hookrightarrow$} \vskip -0.2ex \hbox{\hskip.3ex${\mathsmaller{\prec}}$}}}
T}
\def\lleq{\mathrel{\vcenter{\offinterlineskip\hbox{$\ll$} \vskip -1.4ex\hbox{\hskip0.25ex$\underline{\phantom{<}}$}}}}
\def\ggeq{\mathrel{\vcenter{\offinterlineskip\hbox{$\gg$} \vskip -1.4ex\hbox{\hskip0.25ex$\underline{\phantom{<}}$}}}}
\def\llleq{\mathrel{\vcenter{\offinterlineskip\hbox{$\lll$} \vskip -1.4ex\hbox{\hskip0.3ex$\underline{\phantom{\ll}}$}}}}

\newcommand\Sol[1]{ {{#1}:T
\mathrel{\vcenter{\offinterlineskip
\hbox{\hskip.4ex${\mathsmaller{\bm \lambda}}$}\vskip 0.2ex \hbox{$\leadsto$} \vskip -0.2ex \hbox{\hskip.3ex${\mathsmaller\prec}$}}}\mathfrak W}}

\newcommand\LongSol[2]{{#1}:T
\mathrel{\vcenter{\offinterlineskip
\hbox{\hskip.4ex${\mathsmaller{\bm \lambda}}$}\vskip 0.2ex \hbox{$\leadsto$} \vskip -0.2ex \hbox{\hskip.3ex${\mathsmaller\prec}$}}}
\mathfrak W\wedge |#1|>#2}

\newcommand\PreSol[2]{ {#1}:T
\mathrel{\vcenter{\offinterlineskip
\hbox{\hskip.4ex${\mathsmaller{\bm \lambda}}$}\vskip 0.2ex \hbox{$\leadsto$} \vskip -0.2ex \hbox{\hskip.3ex${\mathsmaller\prec}$}}}
\mathfrak W\mathrel |{#2}}

\newcommand\LongSolX[4]{{#1}:T
\mathrel{\vcenter{\offinterlineskip
\hbox{\hskip.4ex${\mathsmaller{\bm \lambda}}$}\vskip 0.2ex \hbox{$\leadsto$} \vskip -0.2ex \hbox{\hskip.3ex${\mathsmaller{#3}}$}}}
\mathfrak W\mathrel |#4\wedge |#1|>#2}
\newcommand\PreSolX[2]{{#1}: T
\mathrel{\vcenter{\offinterlineskip
\hbox{\hskip.4ex${\mathsmaller{\bm \lambda}}$}\vskip 0.2ex \hbox{$\leadsto$} \vskip -0.2ex \hbox{\hskip.3ex${\mathsmaller X}$}}}
{\mathfrak W}\mathrel |{#2}}

\newcommand\SolFun[2]{ #1\simeq \{{T
\mathrel{\vcenter{\offinterlineskip
\hbox{\hskip.4ex${\mathsmaller{\bm \lambda}}$}\vskip 0.2ex \hbox{$\leadsto$} \vskip -0.2ex \hbox{\hskip.3ex${\mathsmaller\prec}$}}}
\mathfrak W}\}_{ {#2}}}

\def\Provfor{{\tt IPC}^\prec}
\def\provcode{{\tt Proof}^\prec_T}
\def\provcodeX{{\tt Proof}_X}
\def\transin{{\tt TI}}
\def\compax{{\tt CA}}
\def\provfor{{\tt Proof}}

\newcommand\limfor[1]{#1\simeq{\tt Lim}\{{T
\mathrel{\vcenter{\offinterlineskip
\hbox{\hskip.4ex${\mathsmaller{\bm \lambda}}$}\vskip 0.2ex \hbox{$\leadsto$} \vskip -0.2ex \hbox{\hskip.3ex${\mathsmaller\prec}$}}}
\mathfrak W}\}}

\newcommand\limvec{\overrightarrow{\tt Lim}\{ {T
\mathrel{\vcenter{\offinterlineskip
\hbox{\hskip.4ex${\mathsmaller{\bm \lambda}}$}\vskip 0.2ex \hbox{$\leadsto$} \vskip -0.2ex \hbox{\hskip.3ex${\mathsmaller\prec}$}}}
\mathfrak W}\}}

\newcommand\nlimfor[1]{#1\not\simeq{\tt Lim}\{{T
\mathrel{\vcenter{\offinterlineskip
\hbox{\hskip.4ex${\mathsmaller{\bm \lambda}}$}\vskip 0.2ex \hbox{$\leadsto$} \vskip -0.2ex \hbox{\hskip.3ex${\mathsmaller\prec}$}}}
\mathfrak W}\}}

\def\lb{\left\llbracket}
\def\rb{\right\rrbracket}

\newcommand{\rulefor}[5]{{\ensuremath{{\tt Rule}_{#1}^{#4}(#2,#3 \mathrel |#5)}}\xspace}

\newcommand{\cons}[1]{\langle{#1}\rangle^\prec_T}
\newcommand{\prov}[1]{[{#1}]^\prec_T}
\newcommand{\provx}[2]{[{#1}]_{#2}}
\newcommand{\consx}[2]{\langle{#1}\rangle_{#2}}

\newcommand{\glp}{{\ensuremath{\mathsf{GLP}}}\xspace}

\usepackage{xspace}

\newcommand{\pa}{\ensuremath{{\mathrm{PA}}}\xspace}

\newcommand{\gl}{{\ensuremath{\mathsf{GL}}}\xspace}

\newcommand{\ea}{\ensuremath{{\rm{EA}}}\xspace}

\newcommand{\base}{\ensuremath{{\mathrm{ACA}}_0}\xspace}
\newcommand{\la}{\langle}
\newcommand{\ra}{\rangle}

\def\lb{\left\llbracket}
\def\rb{\right\rrbracket}

\def\fmodels{\xymatrix{
\ar@{|=}[r]^{<\omega}&
}
}
\def\nmodels{\xymatrix{
\ar@{|=}[r]^{N}&
}
}
\def\<{\left <}

\def\nc{{\Box}}
\def\ps{{\Diamond}}

\def\mlang{{\sf L}_{[\cdot]}}
\def\alang{{\sf L}^2_\forall}
\def\peq{\preccurlyeq}
\def\seq{\succcurlyeq}
\def\>{\right >}

\DeclareSymbolFont{AMSb}{U}{msb}{m}{n}
\DeclareMathSymbol{\N}{\mathbin}{AMSb}{"4E}
\DeclareMathSymbol{\Z}{\mathbin}{AMSb}{"5A}
\DeclareMathSymbol{\R}{\mathbin}{AMSb}{"52}
\DeclareMathSymbol{\Q}{\mathbin}{AMSb}{"51}
\DeclareMathSymbol{\I}{\mathbin}{AMSb}{"49}
\DeclareMathSymbol{\C}{\mathbin}{AMSb}{"43}

\newcommand{\comment}[2]{{\color{red}{\bf #2}}}

\begin{document}

\title{The omega-rule interpretation of transfinite provability logic}
%%%%%%%%%%%%%%%%%JSL
%\author{David Fern\'{a}ndez-Duque${}^{\dag}$ \and Joost J. Joosten${}^{\dag\dag}$}
%\address{\dag Group for Computational Logic\\ Universidad de Sevilla\\dfduque@us.es\\ \dag \dag Department of Logic, History and Philosophy of Science,\\ University of Barcelona %\\ jjoosten@ub.edu}
%%%%%%%%%%%%%%%%%JSL

%%%%%%%%%%%%%%%%%not JSL
\author{David Fern\'andez Duque,\\
Universidad de Sevilla\\
\ \\
Joost J. Joosten\\
Universitat de Barcelona}
%%%%%%%%%%%%%%%%%not JSL

\maketitle%not jsl

%%%%%%%%%%%%%%%%%%%%%%%%%%%%%%%%%%%%%%%%%%%%%%%%%%%%%%%%%%%%%%%%%%%%%%%%%%%%
\begin{abstract}
Given a recursive ordinal $\Lambda,$ the transfinite provability logic $\glp_\Lambda$ has for each $\xi < \Lambda$ a modality $[\xi]$ with the intention of representing a sequence of provability predicates of increasing strength. One possibility is to read $[\xi]\phi$ as {\em $\phi$ is provable in $T$ using an $\omega$-rule of depth $\xi$,} where $T$ is a second-order theory extending \base. 

In this paper we shall fomalize this notion in second-order arithmetic. Our main results are that, under some fairly general conditions for $T$, the logic $\glp_\Lambda$ is sound and complete for the resulting interpretation.
\end{abstract}

\section{Introduction}

One compelling and particularly successful interpretation of modal logic is to think of $\nc\phi$ as {\em the formula $\phi$ is provable,} where {\em provability} is understood within a formal theory $T$ capable of coding syntax. This was suggested by G\"odel; indeed, if we use $\ps\phi$ as a shorthand for $\neg\nc\neg\phi$, the Second Incompleteness Theorem could be written as $\ps\top\rightarrow\ps\nc\bot.$ It took some time, however, for a complete set of axioms to be assembled, namely until L\"ob showed $\nc(\nc\phi\to\phi)\to\nc\phi$ to be valid. It took longer still for the resulting calculus to be proven complete by Solovay \cite{Solovay:1976}. The resulting modal logic is called $\mathsf{GL}$ (for G\"odel-L\"ob).

Later, Japaridze \cite{Japaridze:1988} enriched the language of \gl by adding a sequence of provability modalities $[n]$, for $n<\omega$. The modality $[0]$ is now used as before to state that $\phi$ is derivable within some fixed formal theory $T$, while higher modalities represent provability in stronger and stronger theories. There are many arithmetic interpretations for Japaridze's logic, and one of them also stems from an idea of G\"odel, who introduced the notion of a theory $T$ being {\em $\omega$-consistent}: $T$ is {\em $\omega$-consistent}  whenever for any formula $\phi$, if $T\vdash \phi(\bar n)$ for all $n\in \mathbb{N}$, then $T\nvdash \exists x \neg \phi(x)$. Dually to this notion one can define a notion of $\omega$-provability: $\phi$ is $\omega$-provable in $T$ whenever $T+\neg \phi$ is $\omega$-inconsistent.  One may then interpret $[1]\phi$ as {\em $\phi$ is $\omega$-provable;} a detailed discussion of this is given in Boolos \cite{Boolos:1993:LogicOfProvability}. 

One can then go on to interpret the higher modalities by using iterated $\omega$-rules. This idea was already explored by Japaridze and gives an interpretation for which the polymodal logic $\glp_\omega$ is sound and complete; Ignatiev \cite{Ignatiev:1993:StrongProvabilityPredicates} and Beklemishev \cite{Beklemishev:2011:SimplifiedArithmeticalCompleteness} later improved on this result. This logic is much more powerful than $\mathsf{GL}$, and indeed Beklemishev has shown how it can be used to perform an ordinal analysis of Peano Arithmetic and its natural subtheories \cite{Beklemishev:2004:ProvabilityAlgebrasAndOrdinals}.

Our (hyper)arithmetical interpretations will be a straightforward generalization of Japaridze's where we read $[\alpha]_T\phi$ as {\em The formula $\phi$ is derivable in $T$ using $\omega$-rules of nesting depth at most $\alpha$;} we shall make this precise later. We do this by considering a well-ordering $\prec$ on the naturals and defining a logic $\glp_\prec$. This is a variation of $\glp_\Lambda$ already studied by the authors and Beklemishev \cite{Beklemishev:2005:VeblenInGLP,FernandezJoosten:2012:ModelsOfGLP}, with the sole difference that we shall represent ordinals as natural numbers rather than appending them as external entities.

Our main result is that $\glp_\prec$ is sound and complete for arithmetical interpretations on `suitable' theories $T$; we will mainly work with extensions of ${\rm ACA}_0$, but as we shall discuss later, it is possible to work over a weaker base theory.

\paragraph{Plan of the paper} Section \ref{syntax} gives a quick review of the logics $\glp_\prec$ as well as their Kripke semantics, and Section \ref{SecSOA} of second-order arithmetic. Section \ref{SecNestedOmega} formalizes the notion of iterated $\omega$-provability in second-order arithmetic; this notion is most naturally interpreted in {\em introspective theories,} introduced in Section \ref{SecIntTheo}. Section \ref{secsound} proves that $\glp_\prec$ is sound for our interpretation. In order to prove completeness, Section \ref{jsec} gives a brief review of the modal logic $\sf J$, which is used in the completeness proof provided in Section \ref{SecCom}. Finally, Appendix \ref{AppBaseTheory} discusses our choice of base theory and Appendix \ref{section:AlternativePresentations} possible variations on the notion of iterated $\omega$-provability.

%%%%%%%%%%%%%%%%%%%%%%%%%%%%%%%%%%%%%%%%%%%%%%%%%%%%

\section{The logic $\mathsf{GLP}_\prec$}\label{syntax}

Formulas of the language $\mlang$ are built from $\bot$ and countably many propositional variables ${p}\in{\mathbb P}$ using Boolean connectives $\neg,\wedge$ and a modality $[\xi]$ for each natural number $\xi$. As is customary, we use $\<\xi\>$ as a shorthand for $\neg[\xi]\neg$.

If $\prec$ is a binary relation on the naturals, the logic $\mathsf{GLP}_\prec$ is given by the following rules and axioms:
\begin{enumerate}
\item all propositional tautologies{,}
\item $[\xi](\phi\to\psi)\to([\xi]\phi\to[\xi]\psi)$ for all $\xi${,}
\item {$[\xi]([\xi]\phi \to \phi)\to[\xi]\phi$ for all $\xi$}{,}
\item $\<\zeta\>\phi\to\<\xi\>\phi$ for $\xi\prec\zeta${,}
\item $\<\xi\>\phi\to [\zeta]\<\xi\>\phi$ for $\xi\prec\zeta$,
\item Modus Ponens, Substitution and Necessitation: $\displaystyle \dfrac{\phi}{[\xi]\phi}$.
\end{enumerate}
We will normally be interested in the case where $\prec$ is a well-order, in which case it is known that $\langle \xi\rangle\top$ is consistent with $\glp_\prec$ for all $\xi$ (see \cite{FernandezJoosten:2012:ModelsOfGLP}). In case $\prec$ is a recursive well-order of order-type $\Lambda$ we shall often write $\glp_\Lambda$ instead of $\glp_\prec$ making the necessary definitional changes for finite order types $\Lambda$.
% I see that you removed my remark "making the necessary definitional changes for finite order types $\Lambda$." but since I mention GLP_2 somewhere in the paper, the remark is necessary. Remember that we have for each natural number xi a modality [xi] which is difficult if you wish to work with finite order types. I do not want to go into details, but do wish to be correct, that's why I included 10 words here.

We shall also work with Kripke semantics. A {\em Kripke frame} is a structure $\mathfrak F=\<W,\<R_i\>_{i<I}\>$, where $W$ is a set and $\<R_i\>_{i<I}$ a family of binary relations on $W$. A {\em valuation} on $\mathfrak F$ is a function $\lb\cdot\rb:\mlang\to \mathcal P(W)$ such that
\[
\begin{array}{lcl}
\lb\bot\rb&=&\varnothing\\\\
\lb\neg\phi\rb&=&W\setminus\lb\phi\rb\\\\
\lb\phi\wedge\psi\rb&=&\comment{W}{}\lb\phi\rb\cap\lb\psi\rb\\\\
\lb\<i\>\phi\rb&=&R^{-1}_i\lb\phi\rb.
\end{array}
\]
A {\em Kripke model} is a Kripke frame equipped with a valuation $\lb\cdot\rb$. Note that propositional variables may be assigned arbitrary subsets of $W$. Clearly, a valuation is uniquely determined once we have fixed its values for the propositional variables.
%
%Often we will write $\<\mathfrak F,x\>\models\psi$ instead of $x\in\lb \psi\rb$. 
% JjJ NOTE: It is funny, in the whole paper there was not a single occurrence of this notation, other than here :-) So, I have omitted this sentence JjJ
%
%
As usual, $\phi$ is {\em satisfied} on $\la \mathfrak F, \lb\cdot\rb\ra$ if $\lb\phi\rb\not=\varnothing$, and {\em valid} on $\la \mathfrak F, \lb\cdot\rb\ra$ if $\lb\phi\rb=W$.

It is well-known that polymodal $\mathsf{GL}$ is sound for $\mathfrak F$ whenever $R^{-1}_i$ is well-founded and transitive, in which case we write $R^{-1}_i$ as $<_i$. However, constructing models of $\glp_\Lambda$ is substantially more difficult than constructing models of $\mathsf{GL}$, since the full logic $\mathsf{GLP}_\Lambda$ is not sound and complete for any class of Kripke frames. In Section \ref{jsec} we will circumvent this problem by working in Beklemishev's $\sf J$, a slightly weaker logic that is complete for a manageable class of Kripke frames.

%%%%%%%%%%%%%%%%%%%%%%%%%%%%%%%%%%%%%%%%%%%%%%%%%%%%%%%%%%

\section{Second-order arithmetic}\label{SecSOA}

Aside from the modal language $\mlang$, we will work mainly in the language $\alang$ of second-order arithmetic.

We fix some primitive recursive G\"odel numbering mapping a formula $\psi\in\alang$ to its corresponding G\"odel number $\ulcorner \psi \urcorner$, and similarly for terms and sequents of formulas (used to represent derivations). Moreover, we fix some set of \emph{numerals} which are terms so that each natural number $n$ is denoted by exactly one numeral written as $\overline n$. Since we will be working mainly inside theories of arithmetic, we will often identify $\psi$ with $\ulcorner \psi \urcorner$ or even with $\overline{\ulcorner \psi \urcorner}$ for that matter.

Our results are not very sensitive to the specific choice of primitive sym\-bols; however, to simplify notation, we will assume we have the following terms available:

\begin{enumerate}
\item A term $\langle x,y\rangle$ which returns a code of the ordered pair formed by $x$ and $y$.

\item A term $x[y/z]$ which, when $x$ codes a formula $\phi$, $y$ a variable $v$ and $z$ a term $t$, returns the code of the result of subsituting $t$ for $v$ in $\phi$. Otherwise, its value is unspecified, for example it could be the default $\ulcorner\bot\urcorner$. We shall often just write $\phi (t)$ for this term if the context allows us to.

\item A term $x\to y$ which, when $x,y$ are codes for $\phi,\psi$, returns a code of $\phi\to\psi$, and similarly for other Boolean operators or quantifiers. The context should always clarify if we use the symbol $\to$ as a term or as a logical connective.

\item A term $\overline x$ mapping a natural number to the code of its numeral.

\item For every formula $\phi$, a term $\phi(\dot x)$ which, given a natural number $n$, returns the code of the outcome of $\phi[x/\bar n]$, i.e., the code of $\phi(\overline n)$.
\end{enumerate}

We will also use this notation in the metalanguage. The only purpose of assuming these terms exist is to shorten complex formulas, as the graphs of all these functions are definable by low level arithmetic formulas over most standard arithmetic languages.

As is customary, we use $\Delta^0_0$ to denote the set of all formulas (possibly with set parameters) where all quantifiers are ``bounded'', that is, of the form $\forall x<y \ \phi$ or $\exists x<y \ \phi$. We simultaneously define $\Sigma^0_{0}=\Pi^0_{0}=\Delta^0_0$ and $\Sigma^0_{n+1}$ to be the set of all formulas of the form $\exists x_0\hdots\exists x_n \phi$ with $\phi\in \Pi^0_{n}$ and similarly $\Pi^0_{n+1}$ to be the set of all formulas of the form $\forall x_0\hdots\forall x_n \phi$ with $\phi\in \Sigma^0_{n}$. We denote by $\Pi^0_\omega$ the union of all $\Pi^0_n$; these are the {\em arithmetic formulas.}

The classes $\Sigma^1_n,\Pi^1_n$ are defined analogously but using second-order quantifiers and setting $\Sigma_0^1 = \Pi^1_0 = \Delta^1_0 = \Pi^0_\omega$. It is well-known that every second-order formula is equivalent to another in one of the above forms.  If $\Gamma$ is a set of formulas, we denote by $\hat\Gamma$ the subset of $\Gamma$ where no set variables appear free.

We will say a theory $T$ is {\em representable} if there is a $\hat\Delta^0_0$ formula ${\provfor}_T(x,y)$ which holds if and only if $x$ codes a derivation in $T$ of a formula coded by $y$; in general we assume all theories to be representable. We may also assume without loss of generality that any derivation $d$ is a derivation of a unique formula $\phi$, for example by representing $d$ as a finite sequence of formulas whose last element is $\phi$. Also, we assume that every formula that is derivable has arbitrarily large derivations; this is generally true of standard proof systems, for example one may add many copies of an unused axiom or many redundant cuts. Whenever it does not lead to confusion we will work directly with codes rather than formulas; if $\phi$ is a natural number (supposedly coding a formula) we use $\Box_T\phi$ as a shorthand for $\exists y\ {\provfor}_T(y, {\phi})$.\\
\medskip

It is important in this paper to keep track of the second-order principles that are used; below we describe the most important ones. We use $<$ to denote the standard ordering on the naturals and $\Gamma$ denotes some set of formulas:
\begin{center}
\begin{tabular}{lll}

$\Gamma$-$\compax$&$\exists X\forall x(x\in X\leftrightarrow \phi(x))$&where $\phi\in\Gamma$;\\

${\tt I}$-$\Gamma$&$\phi(\overline 0)\wedge\forall x(\phi(x)\to\phi(x+\overline 1))\to\forall x\phi(x)$&where $\phi\in\Gamma$;\\

${\tt Ind}$&$\forall x (\,\forall \, y{<}x\, y\in X\rightarrow x\in X)\to\forall x\, x{\in} X$.&\\
\end{tabular}
\end{center}
We assume {\em all} theories extend two-sorted first-order logic, so that they include Modus Ponens, Generalization, etc., as well as Robinson's Arithmetic, i.e. Peano Arithmetic without induction.

Another principle that will be relevant to us is {\em transfinite recursion,} but this is a bit more elaborate to describe. For simplicity let us assume that $\alang$ contains only monadic set-variables; binary relations and functions can be represented by coding pairs of numbers. It will be convenient to establish a few conventions for working with binary relations in second-order arithmetic. First, let us write {\em $R$ is a binary relation} and {\em $f$ is a function:}
\begin{align*}
{\tt rel}(R)&=\forall x\Big(x\in R\to\exists y\exists z\big(x=\langle y,z\rangle\big)\Big),\\
{\tt funct}(f)&={\tt rel}(f)\wedge\forall x\exists! y(\langle x,y\rangle\in f).
\end{align*}
Here, $\exists!$ is the standard abbreviation for {\em there exists a unique.}

Also for simplicity, we may write $n\mathrel R m$ if $R$ represents a relation and $\langle n,m\rangle\in R$ as well as $n\not\mathrel R m$ for $\neg(\langle n,m\rangle\in R)$, or $n=f(m)$ if $\langle m,n\rangle\in f$ and $f$ is meant to be interpreted as a function. Further, it is possible to work with a second-order equality symbol, but it suffices to define $X\equiv Y$ by $\forall x(x\in X\leftrightarrow y\in Y)$.

It will also be important to represent ordinals in second-order arithmetic. For this we will reserve a set-variable $\prec$. Here, we will need to express {\em the relation $\prec$ is a linear order} and {\em the relation $\prec$ is well-ordered,} as follows:\\

\noindent ${\tt linear}(\prec):$
\[
\forall x\big(\neg(x\prec x)\wedge\forall y(x\prec y\vee y\prec x \vee y=x)\big)
\wedge \forall \, x,y,z\ (x{\prec} y \wedge y {\prec} z \to x {\prec} z);
\]
\noindent ${\tt wo}(\prec):$
\[
{\tt linear}(\prec)\wedge\forall X\,\Big(\exists x(x\in X)\to \exists y\forall z\big(z\prec y\to\neg (z\in X)\big)\Big).
\]
We will use Greek letters for natural numbers when viewed as ordered under $\prec$. When it is clear from context we may use natural numbers to represent finite ordinals, so that, for example, $0$ is the least element under $\prec$, independently of whether it truly corresponds to the natural number zero.

We shall often want that the elementary properties of $\prec$ be provable, for example,
\[
\xi \prec \zeta \ \rightarrow \ \Box_T \ \ \xi \prec \zeta.
\] 
This can be guaranteed if we work with recursive well-orders, in which case we assume $T$ contains an axiom $\forall\, x,y \ (x\prec y \leftrightarrow \sigma(x,y))$ for some $\hat \Sigma_1^0$ formula $\sigma(x,y)$.

{\em Transfinite recursion} is the principle that sets may be defined by iterating a formula along a well-order. To formalize this, let us consider a set $X$ whose elements are of the form $\langle \xi,x\rangle$. Write $X_\xi$ for $\{x \mid \langle \xi,x\rangle\in X\}$ and $X_{\prec\xi}$ for $\{x \mid \exists\, \zeta{\prec}\xi\ \langle \zeta,x\rangle \in X\}$. Then, given a set of formulas $\Gamma$ we define\\

\begin{center}
\begin{tabular}{lll}
${\tt TR}\text{-}\Gamma$&${\tt wo}(\prec)\rightarrow\exists X\forall \xi\forall x \ \ \Big ( x\in X_\xi\leftrightarrow \phi(x,X_{\prec\xi})\Big)$ & for $\phi\in\Gamma.$
\end{tabular}
\end{center}

With this, we may define the following systems of arithmetic:

\begin{center}
\begin{tabular}{llllr}
${\rm RCA}_0$&:=&${\tt I}$-$\Sigma^0_1$&+&$\Delta^0_0$-$\compax$\\
${\rm ACA}_0$&:=&${\tt Ind}$&+&$\Pi^0_\omega$-$\compax$\\
${\rm ATR}_0$&:=&$\tt Ind$&+&${\tt TR}$-$\Pi^0_\omega$.
\end{tabular}
\end{center}
We list these from weakest to strongest, but even ${\rm ATR}_0$ is fairly weak in the realm of second-order arithmetic. For convenience we will work mainly in \base, but later discuss how our techniques could be pushed down even to below ${\rm RCA}_0$ at the cost of slightly stronger transfinite induction. 

The system ${\rm ATR}_0$ is relevant because we will define iterated provability by recursion over the well-order $\prec$. However, as we shall see, we require much less than the full power of arithmetic transfinite recursion.\\

In various proofs we wish to reason by transfinite induction. By ${\transin}(\prec,\phi)$ we denote the transfinite induction axiom for $\phi$ along the ordering $\prec$:
\[
{\transin}(\prec,\phi) \ := \forall \xi \ (\forall \, \zeta{\prec}\xi \ \phi(\zeta)\to \phi(\xi)) \to \forall \xi \phi (\xi).
\]
We will write $\phi$-${\compax}$ instead of $\{\phi\}$-${\compax}$, i.e., the instance of the comprehension axiom stating that $\{ x \mid \phi(x) \}$ is a set. The following lemma tells us that we have access to transfinite induction for formulas of the right complexity:

\begin{lemma}\label{theorem:WellOrderingWithComprehensionImpliesTI}
In any second order arithmetic theory containing predicate logic we can prove
\[
{\tt wo}(\prec) \wedge \neg\phi\text{-}{\compax} \to {\transin}(\prec, \phi).
\]
\end{lemma}

\AlternativeProof{}

\CurrentProof{
\begin{proof}
Reason in $T$ and assume ${\tt wo}(\prec) \wedge \neg\phi\text{-}{\compax}$. We prove ${\transin}(\prec, \phi)$ by contraposition. Thus, suppose that $\exists \lambda \neg \phi(\gamma)$. As $\{ \xi \mid \neg \phi(\xi) \}$ is a set, we can apply ${\tt wo}(\prec)$ to obtain the minimal such $\lambda$. Clearly for this minimal $\lambda$ we do not have $\forall \, \zeta{\prec}\lambda \ \ \phi(\zeta)\to \phi(\lambda)$.
\end{proof}
}

%%%%%%%%%%%%%%%%%%%%%%%%%%%%%%%%%%%%%%%%%%%%%%%%%%%%%%%%%%

\section{Nested $\omega$-rules}\label{SecNestedOmega}

In this section we shall formalize the notion of iterated $\omega$-rules inside second-order arithmetic. In Boolos (\cite{Boolos:1993:LogicOfProvability}) it is noted that multiple {\em parallel} applications of the $\omega$-rule do not add extra strength. For example, the rule that allows us to conclude $\sigma$ from 
\[
\begin{array}{lll}
\forall n & \vdash & \psi (\overline n)\\
\forall m & \vdash & \forall x \psi (x) \to \phi(\overline m)\\
 & \vdash & \forall x \, \phi (x) \ \to \ \sigma\\
\end{array}
\]
can actually be derived by a single application of the $\omega$ rule.

However, when we admit slightly less uniformity by allowing $\psi$ to depend on $m$ in this rule, and adding the premises $\forall n  \vdash  \psi_m (\overline n)$ we get our notion of $2$-provability. More generally, we may iterate this process to generate a hierarchy or stronger and stronger notions of $\xi$-provability for a recursive ordinal $\xi$. It is the \emph{nesting depth} that gives extra strength and not the number of applications.

We will use $\prov\lambda\phi$ to denote our representation of  {\em The formula $\phi$ is provable in $T$ using one application of an $\omega$-rule of depth $\lambda$ (according to $\prec$).}
The desired recursion for such a sequence of provability predicates is given by the following equivalence.\footnote{There are other reasonable ways of defining this recursion. In Appendix \ref{section:AlternativePresentations}, we shall discuss some possible alternatives.
}
\begin{equation}\label{PendejoUseShortLabels}
\prov\lambda\phi \ \leftrightarrow 
\ \Big( \Box_T \phi \vee \exists \, \psi\, \exists\, \xi{\prec} \lambda \ \big(\forall n \ \prov\xi\psi(\dot{n}) \ \wedge \ \Box_T (\forall x \psi (x) \to \phi) \big) \Big).
\end{equation}
As a first step in such a formalization, we will use a set $X$ as an `iterated provability class' IPC for short. Its elements are codes of pairs $\langle\lambda,\phi\rangle$, with $\lambda$ a code for an ordinal and $\phi$ a code for a formula; we use $\provx\lambda X \phi$ as a shorthand for $\langle \lambda,{ \phi}\rangle\in X$ and $\consx \lambda X\phi$ for $\langle\lambda,{\neg\phi}\rangle\not\in X$. Clearly, any IPC  will depend on a parameter $\prec$ whose intended interpretation is a well-ordering on the naturals. We then define a formula $\Provfor_T (X)$ (`iterated provability class') as a formalization of:\\

\noindent {\em $\provx \lambda X \phi$ if and only if
\begin{enumerate}
\item 
$\lambda=0$ and $\Box_T\phi$, or

\item there is a formula $\psi(x)$ and an ordinal $\xi\prec\lambda$ such that

\begin{enumerate}
\item for each $n<\omega$, $\provx\xi X{\psi(\overline n)}$, and

\item $\Box_T(\forall x\psi(x)\to\phi)$. 
\end{enumerate}
\end{enumerate}}
Intuitively, we understand $\Provfor_T(X)$ as stating ``$X$ is an iterated provability predicate'' and in the remainder of this text we will use both `class' or `predicate' to refer to IPCs. Let us enter in a bit more detail:

\begin{definition}
Define $\rulefor T \lambda \phi \prec X$ by
\[
\exists \psi \, \exists \, \xi{\prec}\lambda\ 
\big(\forall n\, \provx \xi X{\psi(\dot n)} \wedge \Box_T(\forall x \psi(x)\to \phi)\Big)
\]
and let $\Provfor_T(X)$ be the formula 
\[
\forall z \ \Big[ z\in X\leftrightarrow\exists\lambda\exists \phi\, \Big( z=\langle\lambda,\phi \rangle 
\wedge(\Box_T\phi
\vee \rulefor T \lambda \phi \prec X \Big)\Big].
\]

Then, $\prov\lambda\phi$ is the $\Pi^1_1$-formula $\forall X(\Provfor_T(X)\to \provx \lambda X{\phi})$. 
\end{definition}

Note that the formulas $\provx  \lambda X\phi$ and $\consx  \lambda X\phi$ are  independent of $T$ and of $\prec$ and are merely of complexity $\Delta^0_0$.
Note also that for r.e. theories $T$ we have that $\rulefor T \lambda \phi \prec X$ is a $\Sigma^0_2$-formula whence $\Provfor_T( X)$ is a $\Pi^0_3$-formula. We can write the definition of $\Provfor_T(X)$ more succinctly as
\[
\Provfor_T(X) \ \leftrightarrow \ \forall \xi, \phi \ \Big(\provx \xi X \phi \ \leftrightarrow \ \big( \Box_T \phi \vee \rulefor T \xi \phi \prec X \big)\Big).
\]
From the definition of our provability predicates we easily obtain monotonicity: 
\begin{lemma}\label{theorem:monotonicity} 
Given theories $U,T$ where $U$ extends \base and $T$ is representable, we have that
\[
U\vdash(\xi\prec\zeta)\to([\xi]^\prec_T\phi \to[\zeta]^\prec_T\phi).
\]
\end{lemma}

\AlternativeProof{
We leave the proof as an exercise.
}

\CurrentProof{
\proof We reason within $U$. Suppose $[\xi]^\prec_T\phi$ holds as well as $\Provfor_T(X)$. Letting $x$ be any variable not appearing in $\phi$ we have that $\provx \xi X \phi(\bar n)$ holds for all $n<\omega$ as well as $\Box_T(\forall x\phi\to\phi)$. Thus, $\provx \zeta X\phi$ and, since $X$ was arbitrary, $\prov\zeta\phi$ holds.
\endproof
}

\begin{corollary}\label{theorem:SoundnessOfNecessitation}
Let $U$ and $T$ be as in Lemma \ref{theorem:monotonicity}. If $T\vdash \phi$, then $T \ \vdash  \ \prov \lambda \phi$.
\end{corollary}

\AlternativeProof{}

\CurrentProof{
\begin{proof}
If $T\vdash \phi$ then $\Box_T \phi$ is a true $\hat \Sigma_1^0$ sentence whence $U\vdash \Box_T \phi$. Thus also $U\vdash \prov 0 \phi$ and by monotonicity (Lemma \ref{theorem:monotonicity}), we get $U\vdash \prov \lambda \phi$.
\end{proof}
}

Our notion of $\xi$-provability, $\prov \xi  { }$, is a very weak one as it has a universal quantification over all provability predicates $X$ and it may be the case that there are no such predicates. Dually, the notion of consistency $\cons \xi  {}$  is very strong as it in particular asserts the existence of a provability predicate. In particular, we always provably have
\begin{equation}\label{equation:ProvabilityImpliesZeroProv}
\Box_T \phi \to \prov 0 \phi.
\end{equation}
However, we can in general not prove $\prov 0\phi \rightarrow \nc_T\phi$. Nevertheless, the two notions of provability coincide under the assumption that a provability predicate exists:

\begin{lemma}\label{zerolemm}
Given a representable theory $T$ and a theory $U$ that extends \base we have 
\[
U \vdash \exists X \Provfor_T(X) \to (\Box_T \phi \ \leftrightarrow \ \prov 0 \phi ).
\]
\end{lemma}

\CurrentProof{\proof The proof is straightforward and is left to the reader.\endproof}

\AlternativeProof{
\begin{proof}
Reasoning within $U$, use the assumption that $\exists X \, \Provfor_T(X)$ to choose a provability predicate $Y$ and suppose $\prov 0 \phi$, so that in particular $\provx 0Y\phi$. Then we have either $\rulefor T 0 \phi \prec Y$ or $\Box_T \phi$. As $\rulefor T \lambda\phi\prec Y$ cannot hold for $\lambda =0$ we get $\Box_T \phi$, giving us one implication.

The other direction does not depend on the existence of a provability predicate, for indeed if $\nc_T\phi$ holds, then by definition any provability $X$ satisfies $\provx 0X\phi$, and since $X$ is arbitrary we obtain $\prov 0\phi$.
\end{proof}
}

In the field of formalized provability one often uses formalized $\Sigma_1^0$ completeness (see e.g. \cite{Boolos:1993:LogicOfProvability}):

\begin{lemma}
Let $T$ be some representable arithmetic theory with induction for all $\hat\Delta_0^0$ formulas (i.e., $\Delta^0_0$ without free set variables) and where exponentiation is provably total. Let $\sigma$ be a $\hat\Sigma^0_1$ formula. We have
\[
T \vdash \sigma \to \Box_T \sigma .
\]
\end{lemma}
Of course, it makes really no sense to speak of provable $\Sigma_1^0$ completeness where we allow free set variables. In particular we cannot apply it to our notion $\provx 0 X \phi$, which is why Lemma \ref{zerolemm} will often be useful.

A useful fact is that $\prov\lambda\phi$ is well-defined in the following sense:

\begin{lemma}\label{provunique}
Given theories $U,T$ where $U$ extens \base and $T$ is representable, we have that $U$ proves
\[
{\tt wo}(\prec)\to\forall X\forall Y\Big(\Provfor_T(X)\wedge{\Provfor_T}(Y)\to\forall x\ (x{\in} X\leftrightarrow x{\in} Y)\Big).
\]
\end{lemma}

\CurrentProof{
\proof
This follows by a simple induction over $\prec$.
\endproof}

\AlternativeProof{
\proof
Reason in $U$ and assume ${\tt wo}(\prec)$, $\Provfor_T(X)$ and $\Provfor_T(Y)$. We need to see that $\forall \lambda \, \forall\phi \ \Big(\langle \lambda, \phi \rangle \in X \leftrightarrow \langle \lambda, \phi \rangle \in Y\Big)$. Suppose for a contradiction that we can find a $\lambda$ where the equivalence does not hold. As \[\{ \lambda \mid \exists \phi \ \neg(\langle \lambda, \phi \rangle \in X \leftrightarrow \langle \lambda, \phi \rangle \in Y) \}\]
is a set, by ${\tt wo}(\prec)$ we can assume $\lambda$ to be minimally such. Now we use the fact that both $X$ and $Y$ are iterated provability operators. By the minimality of $\lambda$, both $X$ and $Y$ agree below $\lambda$. But this contradicts the defining equivalence for being a provability predicate. 
\endproof
}

Moreover, under the assumption of ${\tt wo}(\prec)$ we can also show a useful monotonicity property of our provability predicates.

\begin{lemma}\label{theorem:ProvPredMonotoneInT}
Let $U$ be some theory extending \base and let $T$ and $T'$ be representable theories. Write $T\subseteq T'$ as a shorthand for $\forall \phi\  \nc_T\phi\rightarrow \nc_{T'}\phi$.

Then,
\[
U + {\tt wo}(\prec) \vdash T\subseteq T' \to ( [\lambda]^\prec_T \phi \to [\lambda]^\prec_{T'} \phi).
\]
\end{lemma}

\begin{proof}
Reason in $U + {\tt wo}(\prec)$ and assume that $T\subseteq T'$. By an easy induction on $\lambda$ it is shown that if both $\Provfor_T(X)$ and $\Provfor_{T'}(X')$, then  $\la \lambda, \phi \ra \in X \to \la \lambda, \phi \ra \in X'$.
\end{proof}

%%%%%BOOKMARK 

\section{Introspective theories}\label{SecIntTheo}

For a theory $T$ to be able to reason about non-trivial facts of iterated provability at all, it is necessary for it to at least ``believe'' that such a notion exists. For strong theories this is not an issue, but there is no reason to assume that ${\rm ACA}_0$ or any weaker theory is capable of proving that we have provability predicates. Hence we shall pay attention to those theories that do have them, and we shall call them {\em introspective theories.}

\begin{definition}[Introspective theory]
An arithmetic theory $T$ is {\em $\prec$-introspective} if $T\vdash\exists X\Provfor_T(X)$. 
\end{definition}

We defined provability operators by transfinite recursion, and as such it should be no surprise that ${\rm ATR}_0$ is introspective:

\begin{lemma}
Given an elementarily presented theory $T$,
\[{\rm ATR_0}\vdash {\tt wo}(\prec)\rightarrow\exists X\Provfor_T(\prec,X).\]

In particular, ${\rm ATR}_0$ is introspective.
\end{lemma}

However, we wish to work over much weaker theories than ${\rm ATR}_0$, which may not be introspective. Our strategy will be to consider some sort of an ``introspective closure'', but do not wish for it to become much stronger than the original theory. Fortunately, this is not too difficult to achieve.

\begin{definition}
We define the {\em $\prec$-introspective closure} of $T$ as the theory $\overline T$ given by $T+ \exists X\Provfor_T(X)$.
\end{definition}

Below, we use the term ``G\"odelian'' somewhat informally as being susceptible to G\"odel's second incompleteness theorem; for example, it could be taken to mean sound, representable and extending ${\rm RCA}_0$.

\begin{lemma}\label{theorem:TheoryEquiConsistentWithIntrospectiveClosure}
$T$ is equiconsistent with $\overline T$, provided $T$ is G\"odelian and contains $\hat\Delta^0_0$ comprehension.
\end{lemma}

\proof
Clearly the consistency of $\overline{T}$ implies the consistency of $T$. For the other direction we use that if $T$ is G\"odelian, then $T$ is equiconsistent with $T':=T+\nc_T\bot$. We claim that $T'\vdash \exists X\Provfor_T(X)$ so that $T'\supseteq \overline T$ whence 
\[
\begin{array}{lll}
{\tt Con}(T) & \Rightarrow &{\tt Con}(T') \\
 & \Rightarrow &{\tt Con}(\overline T). \\
\end{array}
\]

Indeed, reasoning within $T$, if $T$ were inconsistent, then $\nc_T\phi$ for every formula $\phi$. It follows that if $X$ is an iterated provability operator, then $\provx\lambda X \phi$ for all $\lambda$ and $\phi$; hence the trivial set consisting of all pairs $\langle\lambda,\phi\rangle$ is an iterated provability operator, and by $\hat\Delta^0_0$ comprehension, it forms a set.
\endproof

There is still a danger of sliding down a slippery-slope, where $\overline T$ is itself not introspective, thus needing to generate a sequence of theories that is each ``introspective over the previous''. Fortunately, this is not the case. In order to show this we need a technical lemma reminiscent of the Deduction Theorem. 

\begin{definition}
Let $X$ be an iterated provability operator, so that $\Provfor_T (X)$ holds. We define the set {\em $X$ given $\theta$} --which we denote by $\{X|\theta\}$-- as 
\[
\la \lambda, \phi \ra \in \{X|\theta\} \ :\Longleftrightarrow \ \la \lambda, \theta \to \phi \ra \in X .
\]
\end{definition}

The technical lemma that we shall now prove tells us in particular that introspection is preserved under taking finite extensions.

\begin{lemma}\label{theorem:ConditionalProvPreds}
Let $U$ be some theory containing $\Delta^0_0$ comprehension, and let $T$ be representable. Then
\[
U\vdash \Provfor_T(Y) \to \Provfor_{T+\theta}(\{Y|\theta\}).
\]
\end{lemma}

\CurrentProof{
\proof
We reason in $U$ and assume $\Provfor_T(Y)$. By $\Delta^0_0$ comprehension we see that $\{Y|\theta\}$ is a set. We need to show that 
$\provx \lambda {\{Y|\theta\}} \phi \leftrightarrow \Box_{T+\theta} \phi \vee \rulefor {T+\theta} {\lambda} \phi \prec {\{ Y|\theta\}}$. Since
\[
\begin{array}{lll}
\provx \lambda {\{ Y|\theta\}} \phi & \leftrightarrow & \provx \lambda Y (\theta \to \phi)\\
 & \leftrightarrow & \Box_T (\theta \to \phi) \vee \rulefor {T} {\lambda} {\theta \to \phi} \prec {Y},
\end{array}
\]
and since $\Box_{T+\theta} \phi \leftrightarrow \Box_T (\theta \to \phi)$, it suffices to show $\rulefor {T+\theta} {\lambda} \phi \prec {\{ Y|\theta\}}\  \leftrightarrow \ \rulefor {T} {\lambda} {\theta \to \phi} \prec {Y}$.

But this follows easily from the tautology
\[\big( \theta \to (\forall x \psi (x) \to \phi)\big) \ \leftrightarrow \ \big(  \forall x\ (\theta \to \psi (x)) \to ( \theta \to \phi )\big) 
\]
and the definition of $\{Y|\theta\}$.
\endproof
}

As a direct consequence of this lemma we see that the introspective closure of a theory is indeed itself introspective.

\begin{lemma}
Using $\Delta^0_0$ comprehension one can show that $\overline T$ is introspective.
\end{lemma}

\proof
By the above Lemma \ref{theorem:ConditionalProvPreds}, if $Y$ is a provability predicate for $T$, then $\{ Y|\exists X \Provfor_T(X)\}$
is a provability predicate for $\overline T$. Moreover, by $\Delta^0_0$ comprehension, it forms a set.
\endproof

\noindent
We conclude that working with introspective theories is not too restrictive:

\begin{corollary}\label{theorem:eachTheoryIsEquiConsistentToIntrospective}
Every G\"odelian theory $T$ is equiconsistent to an introspective theory $\overline T$.
\end{corollary}

Note that in general we may not expect for $U$ to prove $\exists X\Provfor_{T+\theta}(X) \to \exists X\Provfor_T(X)$ since for $\theta = \nc_T \bot$ the antecedent always holds (recall the proof of Lemma \ref{theorem:TheoryEquiConsistentWithIntrospectiveClosure}).

A problem that $\overline T$ has is that it introduces an existential second-order formula, which may make it hard later to control the complexity of the resulting theory. Because of this, it is sometimes more convenient to work with {\em explicitly introspective} theories:

\begin{definition}
Given a formal theory $T$, we define $T^\pi$ over the language $\alang+\{\pi\}$, where ${\pi}$ is a new set-constant and $T^\pi=T+\Provfor_T(\pi)$.
\end{definition}

Much as with $\overline T$, $T^\pi$ is introspective, provided $T$ contains $\Delta^0_0$ comprehension. In fact, we can do a bit better in this case. Recall that, given a class of formulas $\Gamma$, $\hat\Gamma$ denotes the set of those formulas of $\Gamma$ with no open set variables; excluding $\pi$, of course, which is a constant. Then we have the following:

\begin{lemma}\label{theorem:ExtendingLanguageNeedsOnlyParameterFreeComprehensionToShowIntrospectivity}
Using $\hat\Delta^0_0$ comprehension one can prove that $T^\pi$ is introspective.
\end{lemma}

The proof proceeds as before and we omit it. Parameter-free comprehension is very convenient in that it does not ``blow up'', as it cannot be iterated; for example, $\Pi^0_1$ comprehension {\em with} set parameters is equivalent to full arithmetic comprehension, but $\hat\Pi^0_1$ comprehension is not.

Now that we have shown that introspective theories are not such a bad thing to work with, we will employ them freely in the next sections. Introspective theories are capable of reasoning about their own iterated provability; for example, we may prove the desired recursion as stated in \eqref{PendejoUseShortLabels}.

\begin{lemma}\label{theorem:RecursiveEquationProved}
Let $T$ be a theory that extends \base. Then, we have that 
\begin{enumerate}
\item\label{theorem:item:NoIntrospectivityNeeded}
$T \vdash  \Big( \Box_T \phi \vee \exists \, \psi\, \exists\, \xi{\prec} \lambda \ \big(\forall n \ \prov\xi\psi(\dot{n}) \ \wedge \ \Box_T (\forall x \psi (x) \to \phi) \big) \Big)
 \ \to \ \prov\lambda\phi$;

\item
$\overline T\vdash  \prov\lambda\phi \ \leftrightarrow \Big( \Box_T \phi \vee \exists \, \psi\, \exists\, \xi{\prec} \lambda \ \big(\forall n \ \prov\xi\psi(\dot{n}) \ \wedge \ \Box_T (\forall x \psi (x) \to \phi) \big) \Big).
$
\end{enumerate}
\end{lemma}

\begin{proof}
In the first item, we reason in $T$ and need to prove $\prov \lambda \phi$ under the assumption of the antecedent. To this end, we fix some $X$ with $\Provfor_T(X)$ and show $\provx \lambda X \phi$. However, this follows directly from the definition of $X$ being a provability predicate since we can replace $\forall n \ \prov\xi\psi(\dot{n})$ in the antecedent by $\forall n \ \provx \xi X \psi(\dot {n})$.

For the remaining implication in the second item we reason as follows. In case $\lambda =0$ we get the implication by Lemma \ref{zerolemm}. In case $\lambda > 0$, from the definition we see that for any provability predicate $X$ we have 
\[
\provx \lambda X \phi \ \to \Big( \Box_T \phi \vee \exists \, \psi\, \exists\, \xi{\prec} \lambda \ \big(\forall n \ \provx \xi X\psi(\dot{n}) \ \wedge \ \Box_T (\forall x \psi (x) \to \phi) \big) \Big).
\]
From this we obtain
\[
\begin{array}{ll}
\prov \lambda \phi  \to & \forall X\Big[\Provfor_T(X) \to \\
&\Big( \Box_T \phi \vee \exists \psi \exists \xi{\prec} \lambda  \big(\forall n \ \provx \xi X\psi(\dot{n}) \wedge \Box_T (\forall x \psi (x) \to \phi) \big) \Big)\Big],
\end{array}
\]
from which the claim directly follows. 
\end{proof}

\section{Soundness}\label{secsound}

In this section we shall see that indeed ${\sf GLP}_\prec$ is sound for its arithmetic interpretation. In Lemma \ref{theorem:monotonicity} we have already seen the soundness of the monotonicity axiom $[\xi] \phi \to [\zeta] \phi$ for $\xi\prec \zeta$. For the remaining axioms we will transfinite induction over $\prec$ so we define, given a second-order theory $T$, a new theory $T^\prec$ as
\[
T^\prec :=T+{\tt wo}(\prec).
\]
We will assume that $T$ contains \base, although in Appendix \ref{AppBaseTheory} we shall discuss this choice. 
Since introspection is closed under taking finite extensions both $\overline{T}^\prec$ and $\overline{T^\prec}$ are introspective (though not necessarily equivalent); for all our arguments below it is irrelevant which one we use.

Let us first check the soundness of the basic distribution axiom.

\begin{lemma}\label{kaxiom} Given theories $U,T$ where $U$ extends \base and $T$ is representable, then
\[
U^\prec\vdash [\lambda]^\prec_T(\phi_1\to\phi_2)\to(\prov\lambda\phi_1 \to\prov\lambda\phi_2).
\]
\end{lemma}

\CurrentProof{
\proof We reason within $U^\prec$.

Let $X$ be a provability predicate. We shall prove by induction on $\lambda$ that 
\begin{equation}\label{equation:TheKAxiom}
\forall \phi_1,\phi_2\ \Big(\provx \lambda X\phi_1\wedge \provx \lambda X (\phi_1\to\phi_2) \to \provx \lambda X \phi_2\Big).
\end{equation}
Note that by Lemma \ref{theorem:WellOrderingWithComprehensionImpliesTI} we only need $\Sigma^0_1$ comprehension (with set parameters) to have access to this transfinite induction.

So, we assume that $\provx \lambda X\phi_1\wedge \provx \lambda X (\phi_1\to\phi_2)$ and let $\psi_1,\psi_2$ be such that
\begin{enumerate}
\item for each $i=1,2$ there is $\xi_i<\lambda$ such that for all $n<\omega$, $\la \xi_i, \psi_i(\overline n)\ra\in X$,
\item 
$\Box_T(\forall x\psi_1(x)\to\phi_1)$,

\item 
$\Box_T(\forall x\psi_2(x)\to(\phi_1\to\phi_2))$.
\end{enumerate}

By first-order logic we see that
\begin{equation}\label{equation:label}
\Box_T(\forall x(\psi_1(x)\wedge \psi_2(x))\to\phi_2).
\end{equation}

Let $\xi=\max\{\xi_1,\xi_2\}$. By induction on $\xi\prec \lambda$ and several uses of Modus Ponens inside $\provx\xi X$ we obtain for each $n$ that $\provx \xi X{(\psi_1(\overline n)\wedge\psi_2(\overline n))}$. But given that $X$ is an IPC, this shows in combination with \eqref{equation:label} that $\provx \lambda X \phi_2$ and we have shown \eqref{equation:TheKAxiom}.

To conclude the proof, we assume that $[\lambda]^\prec_T(\phi_1\to\phi_2) \, \wedge \, \prov\lambda\phi_1$. Thus, for an arbitrary provability predicate $X$ we have $\provx \lambda X (\phi_1\to\phi_2) \, \wedge \, \provx\lambda X \phi_1$ whence by \eqref{equation:TheKAxiom} also $\provx \lambda X \phi_2$. As $X$ was arbitrary, we obtain $\prov \lambda \phi_2$.
\endproof
}

With our distribution axiom at hand we can now obtain a formalized Deduction Theorem.

\begin{lemma}
Let $U$ be a theory extending \base and let $T$ be representable. We have that
\[
U^\prec \vdash [\lambda]_{T+\theta}^\prec\phi \leftrightarrow [\lambda]^\prec_T (\theta \to \phi).
\]
\end{lemma}

\begin{proof}
If $[\lambda]_T^\prec (\theta \to \phi)$ then, by Lemma \ref{theorem:ProvPredMonotoneInT} we also have $[\lambda]_{T+\theta}^\prec (\theta \to \phi)$. Since clearly $[\lambda]^\prec_{T+\theta}\theta$, by the distribution axiom we get $[\lambda]_{T+\theta}^\prec \phi$.

For the other direction, reason in $U^\prec$ and assume $[\lambda]_{T+\theta}^\prec \phi$. Let $X$ be arbitrary with $\Provfor_T(X)$. By Lemma \ref{theorem:ConditionalProvPreds} we see that  $\Provfor_{T+\theta}(\{X|\theta\})$. Now by the assumption that $[\lambda]_{T+\theta}^\prec \phi$ we see that $\provx \lambda {\{X|\theta\}} \phi$ so that consequently $\provx \lambda {X} (\theta \to \phi)$. 
\end{proof}
So far we have shown that some of the axioms of $\glp_{\mathord \prec}$ are sound for our omega-rule interpretation; L\"ob's axiom and the ``provable consistency'' axiom remain to be checked. For the former, the following lemma will be quite useful.

\begin{lemma}
Extend $\sf GL$ with a new operator $\blacksquare$ and the following axioms  for all formulas $\phi,\mbox{ and }\psi$:
\begin{enumerate}
\item $\vdash\nc\phi\to \blacksquare\phi$,
\item $\vdash \blacksquare(\phi\to\psi)\to(\blacksquare\phi\to \blacksquare\psi)$ and,
\item $\vdash \blacksquare\phi\to \blacksquare\blacksquare\phi$,
\end{enumerate}
and call the resulting system ${\sf GL}^\blacksquare$.

Then for all $\phi$,
\[{\sf GL}^\blacksquare\vdash \blacksquare(\blacksquare\phi\to\phi)\to \blacksquare\phi.\]
\end{lemma}

\proof
It is well-known that $\sf GL$ is equivalent to $\sf K4$ plus the L\"ob Rule:
\[
\frac{\nc\phi\rightarrow\phi}\phi.
\]
Thus it suffices to check that this rule holds for $\blacksquare$. But indeed, assume that ${\sf GL}^\blacksquare\vdash \blacksquare\phi\to\phi$. Then, using $\nc\phi\to\blacksquare\phi$ we obtain $\nc\phi\to\phi$, and by L\"ob's rule (for $\nc$) we see that ${\sf GL}^\blacksquare\vdash\phi$, as desired.
\endproof

Thus to show that $[\lambda]^\prec_T$ is L\"obian for all $\lambda$, we need only show the following:

\begin{lemma}\label{trans}
Given a recursive order $\prec$, theories $U,T$ where $U$ extends \base and $T$ is representable, we have that
\[
U^\prec \vdash \forall\phi \forall \lambda\ \ \prov \lambda \phi \to \prov \lambda \prov{\dot\lambda}{\dot\phi}. 
\]
\end{lemma}

\proof
Reason within $U^\prec$. We assume $\Provfor_T(X)$ and will show by induction on $\prec$ that if $\provx \lambda X\phi$, then $\provx{\lambda} X{\prov{\bar\lambda}{\bar\phi}}$, from which the lemma clearly follows.

The base case, when $\lambda{=}0$, is straightforward. We assume $\provx  0 X \phi$ and by Lemma \ref{zerolemm} we get $\Box_T\phi$ whence $\Box_T\Box_T\bar\phi$ by provable $\hat \Sigma^0_1$-completeness of $U$. Consequently, by applying \eqref{equation:ProvabilityImpliesZeroProv} twice (once under the box) we obtain $\prov 0 \prov {\bar 0} {\bar\phi}$ whence certainly also $\provx 0 X  \prov {\bar 0} {\bar \phi}$. 

Now assume that $\lambda\succ 0$ and there are $\xi\prec\lambda$ and $\psi$ such that for all $n$, $\provx \xi X{\psi(\overline n)}$ and $\Box_T(\forall x\psi(x)\to\phi)$.

By the induction hypothesis on $\xi\prec\lambda$, for every number $n$ we can see that $\provx \xi X{\prov{\bar\xi}{{\bar\psi}(\bar n)}}$. Thus we obtain by one application of the $\omega$-rule that
\begin{equation}\label{someq}
\provx \lambda X{\forall n \prov {\bar\xi}}{{\bar\psi}(\dot n)}.
\end{equation}

Meanwhile, we have that $\provx 0 X \Box_T(\forall x\bar\psi(x)\to\bar\phi)$ from which it follows by monotonicity that
\begin{equation}\label{othereq}
\provx \lambda X{\Box_T(\forall x\bar\psi(x)\to\bar\phi)}.
\end{equation}
Since $\prec$ is recursive we also have that 
\begin{equation}\label{yetothereq}
\provx \lambda X \, {\bar\xi} \prec {\bar\lambda} .
\end{equation}
Putting (\ref{someq}), (\ref{othereq}) and (\ref{yetothereq}) together and bringing the existential quantifiers under the box we conclude that
\[
\provx \lambda X \Big(\exists \psi \, \exists \, \xi{\prec}\lambda\ \big( \forall n \prov \xi {\psi (\dot n)} \wedge \Box_T (\forall x \psi (x) \to \phi)\big)\Big).
\]
By an application under the box of Lemma \ref{theorem:RecursiveEquationProved}.\ref{theorem:item:NoIntrospectivityNeeded} (note that no need of introspection is required) we obtain $\provx \lambda X {\prov{\bar\lambda}{\bar\phi}}$ as was to be proven. 
\endproof

In the proof of the remaining $\glp_\prec$ axiom we will need for the first and only time the assumption that $T$ is introspective. 

\begin{lemma}\label{hardaxiom} 
If $U$ is any theory extending $\base$, $\prec$ is recursive and $T$ is representable and $\prec$-introspective, then
\[
U^\prec\vdash \forall\phi\forall\lambda\forall \xi\prec\lambda \  \cons \xi\phi \rightarrow\prov\lambda{\cons {\dot\xi}\dot\phi}.
\]
\end{lemma}

\proof
We reason in $U^\prec$ and assume that $\xi\prec \lambda$. Let us first see that it is sufficient to show that for an arbitrary provability predicate  $X$ we have $\consx \xi X \phi \rightarrow\provx\lambda X {\cons {\bar\xi}{\bar\phi}}$. 

If we wish to show  $\prov\lambda{\cons {\bar\xi}\bar\phi}$ we pick an arbitrary provability predicate $X$ and set out to prove $\provx\lambda X{\cons {\bar\xi}\bar\phi}$. By the assumption $\cons \xi \phi$ we know that there is some provability predicate $Y$ with $\la \xi, \phi \ra \notin Y$, that is, $\consx \xi Y \phi$. By Lemma  \ref{provunique}, we have that $X\equiv Y$, whence $\consx \xi X \phi$ and indeed,  $\consx \xi X \phi \rightarrow\provx\lambda X {\cons {\bar\xi}\bar\phi}$ suffices to finish the proof.

In view of the above, we will prove $\consx \xi X \phi \rightarrow\provx\lambda X {\cons {\bar\xi}\bar\phi}$ by induction on $\lambda$. For the base case, when $\lambda=1$,  we reason as follows. From $\consx 0X\phi$, we obtain the $\hat \Pi^0_1$ sentence $\Diamond_T\phi$, that is, $\forall n \neg {\tt Proof}_T (n,\bar\phi)$. Since $\neg {\tt Proof}_T( n,\bar\phi) \in \hat \Sigma^0_1$, we get $\forall n \ \Box_T \neg {\tt Proof}_T(\dot n, \bar\phi)$ and also $\forall n \ \prov 0 { \neg {\tt Proof}_T(\dot n, \bar\phi)}$. Then by applying an $\omega$-rule  we see that $\provx 1X{\ps_T \phi}$. Since $T$ is $\prec$-introspective then $\provx 1 X{\exists Y \Provfor_T(Y)}$, and by Lemma \ref{zerolemm}
\[
\provx 1 X{\big(\exists Y\Provfor_T(Y)\rightarrow (\cons {\bar 0}\bar\phi\leftrightarrow \ps_T\bar\phi)\big)},
\]
from which we conclude that $\provx 1X{\cons {\bar 0}\bar\phi}.$

So assume that $\lambda\succ 1$. If we have that $\consx\xi X\phi$ then for every formula $\psi$ either 
\begin{enumerate}
\item 
for all $\eta\prec\xi$ there is $n<\omega$ such that $\consx \eta X {\neg \psi(\overline n)}$, or

\item 
$\Diamond_T(\forall x\psi(x)\wedge\phi)$ holds.
\end{enumerate}

In the first case, by the induction hypothesis for $\xi\prec\lambda$ we can see that $\exists n \ \provx \xi X{\cons{\bar\eta}{\neg \bar\psi(\bar n)}};$ in the second, we have that $\provx \xi X{\Diamond_T(\forall x\bar\psi(x)\wedge\phi)}.$ Combining these, we obtain
\[
\eta \prec \xi \to \provx\xi X{\Big(\exists x\cons{\bar\eta}{\neg \bar\psi( \dot x)}\vee\Diamond_T(\forall x\bar\psi(x)\wedge\bar\phi)\Big)}.
\]
Since $\prec$ is recursive, we know that $\eta\seq\xi \to \provx \xi X \bar\eta\seq\bar\xi$.
We thus see that, for all pairs $\langle\eta, \psi\rangle$,
\[
\provx\xi X{\Big(\bar\eta\seq\bar\xi\vee\exists x\cons{\bar\eta}{\neg \bar\psi( \dot x)}\vee\Diamond_T(\forall x\bar\psi(x)\wedge\bar\phi)\Big)}.
\]

By one application of the $\omega$-rule to all pairs $\langle \eta, \psi\rangle$ (represented as natural numbers) we obtain
\[
\provx \lambda X{\Big(\forall \psi\forall \eta\prec \xi\big( \exists x\cons \eta{\neg \psi(\dot x)}\vee\Diamond_T(\forall x\dot\psi(x)\wedge\bar\phi)\big)\Big)},
\]
and by definition (Lemma \ref{theorem:RecursiveEquationProved}.\ref{theorem:item:NoIntrospectivityNeeded} applied under the box) we get $\provx \lambda X{\cons{\bar\xi}\bar\phi}$.
\endproof

We have essentially proven that $\glp_{\mathord\prec}$ is sound for its omega-rule interpretation, but we need the following definition in order to make this claim precise.

\begin{definition}
An {\em arithmetic interpretation} is a function $f:\mathbb P\to \alang$.

We denote by $f^\prec_T$ the unique extension of $f$ such that $f^\prec_T(p)=f(p)$ for every propositional variable $p$, $f^\prec_T(\bot)=\bot$, $f^\prec_T$ commutes with Booleans and $f^\prec_T([\lambda]\phi)=[\bar\lambda]^\prec_T\ f^\prec_T(\phi)$. 
\end{definition}

\begin{theorem}[Soundness]\label{theorem:soundness}
If $\prec$ is any recursive well-order on the naturals, $U$ is a sound theory extending ${\rm ACA}_0$, $T$ is $\prec$-introspective and representable and ${\sf GLP}_\prec\vdash\phi$ then $U^\prec\vdash f^\prec_T(\phi)$ for every arithmetic interpretation $f$.
\end{theorem}

\proof
By an easy induction on the length of a $\glp_\prec$-proof of $\phi$, using the fact that each of the axioms is derivable. 
%Necessitation follows from $\hat \Sigma^0_1$-completeness (see Corollary \ref{theorem:SoundnessOfNecessitation}).
Necessitation is just Corollary \ref{theorem:SoundnessOfNecessitation}.
\endproof

Now that we have proven that $\glp_\prec$ is sound, our main objective will be to prove the converse of Theorem \ref{theorem:soundness} which is hyper-arithmetical completeness of $\glp_\prec$. For this, let us first review the modal logic $\sf J$.

\section{The logic $\sf J$}\label{jsec}

It is well-known that $\glp_\Lambda$ has no non-trivial Kripke frames for $\Lambda>1$. In order to remedy for this situation, we pass to a weaker logic, Beklemishev's ${\sf J}$. The logic ${\sf J}$ is as $\glp_\omega$ where we replace the monotonicity axiom of $\glp_{\omega}$ by the two axioms

\begin{enumerate}
\item[6.] 
$[n]\to[m][n]\phi$, for $n\leq m$ and

\item[7.] 
$[n]\to[n][m]\phi$, for $n< m$.

\end{enumerate}
The logic ${\sf J}$ is proven in \cite{Beklemishev:2010} to be sound and complete for the class of finite Kripke models $\langle W,\langle>_n\rangle_{n<N},\lb\cdot\rb\rangle$ such that
\begin{enumerate}
\item 
the relations $<_n$ are transitive and well-founded,

\item 
if $n<m$ and $w<_m v$ then $\mathop <_n(w)=\mathop <_n(v)$ (where $\mathop{<_n}(w)=\{u:u<_n w\}$) and,

\item 
if $n<m$ then $w<_m v<_n u$ implies that $w<_n u$.
\end{enumerate}
It will also be convenient to define some auxiliary relations. Say:
\begin{itemize}
\item 
$w\ll_n v$ if for some $m\geq n$, $w<_m v$ and,

\item 
$w\lll_{n} v$ if $w\ll_n v$ or there is $u\in W$ such that $w\ll_{n} u$ and $v\ll_{n+1} u$.
\end{itemize}
By the above frame conditions it is easy to see that $\ll_n$ is transitive and well-founded.

We will also use $\lleq_n$ and $\llleq_n$ to denote the respective reflexive closures. Let $\approx_n$ denote the symmetric, reflexive, transitive closure of $\ll_n$ and let $[w]_n$ denote the equivalence class of $w$ under $\approx_n$. Write $[w]_{n+1}<_n[v]_{n+1}$ if there exist $w'\in[w]_{n+1}$, $v'\in [v]_{n+1}$ such that $w'<_n v'$. %We will make no distinction between a relation and its restriction to some subset.

A $\sf J$-frame $W$ is said to be {\em stratified} if
whenever $[w]_{n+1}<_n[v]_{n+1}$, it follows that $w<_n v$. Note that the property of being stratified in particular entails the modally inexpressible frame condition that $w<_n v$ and $w<_m u$ implies $u<_nv$ whenever $m>n$.
With this we may state the following completeness result also from \cite{Beklemishev:2010}:

\begin{lemma}\label{jcomp}
Any $\sf J$-consistent formula can be satisfied on a finite, stratified $\sf J$-frame.
\end{lemma}

Thus if we can reduce $\glp_\omega$ to $\sf J$, we will be able to work with finite well-behaved Kripke models. For this, given a formula $\phi$ whose maximal modality is $N$, define
\[M(\phi)=\bigwedge_{\substack{[n]\psi\in{\rm sub}(\phi)\\
n< m\leq N}}
[n]\psi\to [m]\psi.
\]
Then we set $M^+(\phi)=M(\phi)\wedge\bigwedge_{n\leq N}[n]M(\phi)$.

The following is also proven in \cite{Beklemishev:2010}:
\begin{lemma}\label{glptoj}
For any formula $\phi\in {\mathcal L}_\omega$, $\glp_\omega\vdash\phi$ if and only if
\[{\sf J}\vdash M^+(\phi)\to \phi.\]
\end{lemma}

We shall use these results in the next section to prove arithmetical completeness by ``piggybacking'' from the completeness of $\sf J$ for finite frames.

%%%%%%%%%%%%%%%%%%%%%%%%%%%%%%%%%%%%%%%%%%%%%%%%%%%%%%%%%%%%%%%%%%%%%%%%%%%%

\section{Completeness}\label{SecCom}

In this section we want to prove that $\glp_\prec$ is complete for its $\omega$-rule interpretation. This means that, given a consistent formula $\phi$, there is an arithmetic interpretation $f$ such that $\neg f^\prec_T(\phi)$ is not derivable in $T$ (we will make this claim precise in Theorem \ref{complete}).

There are many proofs of completeness of $\sf GL$ and ${\sf GLP}_\omega$, and it is possible to go back to an existing proof and adjust it to prove completeness in our setting. Because of this, we should say a few words about our choice of including a full proof in this paper. There are essentially two reasons.

The first is that, while our result follows to a certain degree from known {\em proofs,} it does not follow from known {\em results;} even then, there would be several technical issues in adjusting known arguments to our setting, as they make assumptions that are not available to us.

The second is that the argument we propose carries some simplifications over previous proofs that could also be applied to standard interpretations of $\glp_\omega$, thus contributing to an ongoing effort to find simpler arguments for this celebrated result.

To be more precise, there are at least six proofs in the literature:
\begin{enumerate}
\item Solovay originally constructed a function $h$ with domain $\omega$ of a self\--re\-feren\-tial nature and used statements about $h$ to prove the completeness theorem for the unimodal $\glp_1$ \cite{Solovay:1976}. The proof used the recursion theorem.

\item De Jongh, Jumelet and Montagna introduced a modification using the fixpoint theorem instead of the recursion theorem \cite{JonghJumeletMontagna:1991}, where the function $h$ is simulated via finite sequences that represent computations \cite{JonghJumeletMontagna:1991}. This approach is presented in greater detail in \cite{Boolos:1993:LogicOfProvability}.

\item A more elementary construction using the simultaneous fixpoint theorem is also given in \cite{JonghJumeletMontagna:1991}.

\item Japaridze proved completeness for $\glp_\omega$ with essentially the $\omega$-rule interpretation we are presenting here \cite{Japaridze:1988}.

\item Ignatiev generalized this result to a large family of ``strong provability predicates'' \cite{Ignatiev:1993:StrongProvabilityPredicates}.

\item Beklemishev gave a simplified argument using the logic $\sf J$, which is very well-behaved. However, this proof still considers a family of $N$ Solovay functions $h_n$ with domain $\omega$, where $N$ is the number of modal operators appearing in our ``target formula'' $\phi$.
\end{enumerate}

Despite these strong provability predicates being quite general, they do not apply to our interpretation, as for example it is assumed that they are of increasing logical complexity whereas our iterated provability operators are all given by a single $\Pi^1_1$ formula. The argument we present here, aside from being the fist that considers aribtrary recursive well-orders, combines ideas from \cite{JonghJumeletMontagna:1991} and \cite{Beklemishev:2011:SimplifiedArithmeticalCompleteness} by considering finite paths over a polymodal $\sf J$-frame. We do so by introducing an additional trick, which is to work with all modalities simultaneously, where our path makes a $\lambda_n$-step whenever appropriate. Readers familiar with known proofs might find it surprising that this is not problematic, but indeed it isn't and otherwise the argument proceeds as in other settings. As always, we will mimic a Kripke structure using arithmetic formulas and define our arithmetic interpretation based on them.

Since $\glp_\prec$ is Kripke incomplete, we will resort to $\sf J$-models instead. These models are related to $\glp_\omega$ as described in the previous section.  The step from $\glp_\prec$ to $\glp_\omega$ is provided by the following easy lemma which is also given in \cite{BeklemishevFernandezJoosten:2012:LinearlyOrderedGLP}.

\begin{lemma}\label{TheoPleaseUseShortNames}
Let $\phi$ be a $\glp_\prec$ formula whose occurring modalities in increasing order are $\{ \lambda_0, \ldots, \lambda_N\}$. By $\overline \phi$ we denote the \emph{condensation} of $\phi$ that arises by simultaneously replacing each occurrence of $[\lambda_i]$ by $[i]$. It now holds that
\[
\glp_\prec \nvdash \phi \ \Longrightarrow \ \glp_\omega \nvdash \overline \phi.
\]
\end{lemma}

\begin{proof}
Arguing by contrapositive, given a $\glp_\omega$-derivation $d$ of $\phi$ we may replace every occurrence of $[n]$ in $d$ by $[\lambda_n]$, thus obtaining a derivation of $\phi$.
\end{proof}

With this easy lemma at hand we may give an outline of the proof of our completess theorem, which reads as follows.

\begin{theorem}\label{complete}
If $\prec$ is recursive, $T$ is any sound, representable, $\prec$-introspective theory extending ${\rm ACA}_0$ and proving ${\tt wo}(\prec)$ and $\phi$ is any $\mlang$-formula, ${\sf GLP}_\prec\vdash \phi$ if and only if, for every arithmetical interpretation $f$, $T^\prec\vdash f^\prec_T(\phi)$.
\end{theorem}

\proof[Proof sketch]
One direction is soundness and has already been established.

For the other, if $\glp_\prec \nvdash \phi$ then by our above Lemma \ref{TheoPleaseUseShortNames} combined with Lemma \ref{glptoj}, $M^+(\overline \phi)\wedge\neg \overline \phi$ is not provable in $\sf J$. 

Thus by Lemma \ref{jcomp}, $M^+(\overline \phi)\wedge\neg \overline \phi$ can be satisfied on a world $w_\ast$ of some $\sf J$-model $\mathfrak W' = \langle W', \langle >_n\rangle_{n<N},\lb\cdot\rb\rangle$ where $W'=[1,M]$ for some $M\geq 1$. We construct a new model $\mathfrak W$ which is as $W'$ only that now a new $<_0$-maximal root $0$. The valuation of propositional letters on $0$ is chosen arbitrarily and is irrelevant. 

The next ingredient is to assign to each $w\in W$ an arithmetic sentence $\sigma_w$ so that the formulas $\bm\sigma$ are a ``snapshot'' of $\mathfrak W$. We will make this precise in Definition \ref{defrepre}, but let us outline the essential properties that we need from $\bm\sigma$.

First, we need for the arithmetic interpretation $f$ that sends a propositional variable $p$ to $f(p):=\bigvee_{w\in \lb p\rb}\sigma_w$ to have the property that
\[
\mathfrak W, w\Vdash \overline \psi \ \ \Longleftrightarrow \ \ T \vdash \sigma_w \to f^\prec_T(\psi)
\]
for each $w\in W'$ and each subformula $\psi$ of $\phi$. In particular we have $T \vdash \sigma_{w_\ast} \to \neg f^\prec_T(\phi)$ from which we obtain
\begin{equation}\label{firstNonDescriptiveLabel}
T \vdash \ps_T \sigma_1 \to \neg \Box_T f^\prec_T(\phi) .
\end{equation}
Our desired result will follow if the formulas $\bm\sigma$ satisfy two more properties: the second is that \[T\vdash\sigma_0\rightarrow\ps_T\sigma_1,\] and the third, that $\mathbb N\models\sigma_0$. By the assumption that $T$ is sound we conclude that $\mathbb N \models \neg \Box_T f^\prec_T( \phi)$. Hence, $f^\prec_T(\phi)$ is not provable in $T$ which is what was to be shown.
\endproof

Before we proceed to give the details needed to complete the proof we state as an easy consequence of our arithmetic completeness theorem the following lemma which was also proven by purely modal means in \cite{BeklemishevFernandezJoosten:2012:LinearlyOrderedGLP}.

\begin{corollary}
Given a recursive well-order $\prec$ and an $\mlang$-formula $\phi$ we have that 
\[
\glp_\prec \vdash \phi \ \Longleftrightarrow \ \glp_\omega \vdash \overline \phi.
\]
\end{corollary}

\begin{proof}
One direction is Lemma \ref{TheoPleaseUseShortNames}. For the other direction, suppose $\glp_\omega \nvdash \overline \phi$. By the proof of Theorem \ref{complete} we find an arithmetical interpretation $f$ so that $\overline\base^\prec \nvdash f^\prec_T(\phi)$. By the soundness theorem (Thm. \ref{theorem:soundness}) we conclude that $\glp_\prec \nvdash \phi$. 
\end{proof}

Before entering into further detail, we first say what it means that a collection of sentences $\bm \sigma = \{ \sigma_0, \ldots, \sigma_k\}$ is a snapshot of a Kripke structure with nodes $\{ 0,\ldots ,k\}$ inside a theory. Most importantly, this means that each world $w$ will be associated with an arithmetic sentence $\sigma_w$ so that this sentence carries all the important information in terms of accessible worlds. 

\begin{definition}\label{defrepre}
Given a sequence \[ \bm{\lambda}=\lambda_0\prec\lambda_1\prec \ldots\prec \lambda_{N-1},\]
a finite $\sf J$-model $\mathfrak W=\langle W,\langle<_n\rangle_{n< N}, \lb\cdot \rb\rangle$ with root $0$, and a formal theory $T$, a family of formulas $\{\sigma_w:w\in W\}$ is a {\em $\bm\lambda$-snapshot} of $\mathfrak W$ in $T$ if
\begin{enumerate}
\item 
$T\vdash\displaystyle\bigwedge_{w\not=v\in W} \neg(\sigma_w\wedge\sigma_v)$,

\item 
$T+\sigma_w\vdash \cons {\bar{\lambda_{n}}}{\sigma_v}$ for all $n<N$ and $v<_n w$,

\item 
for all $n<N$ and for each world $w\neq 0$,\[T+\sigma_w \vdash\prov{\bar{\lambda_{n}}} {\displaystyle\bigvee_{v\lll_n w}\sigma_v}\]

\item 
$\mathbb N\models \sigma_0$.
\end{enumerate}

If $\mathfrak W$, $\bm\sigma$, $\bm\lambda$, $T$ are as above we will write $\snapshot{\bm\lambda}{\bm\sigma}.$
\end{definition}

\begin{lemma}\label{truthlemma}
Suppose that $\snapshot{\bm\lambda}{\bm\sigma}$, $\phi$ is an $\mlang$-formula with modalities amongst $\bm\lambda$ such that $\mathfrak W\models M^+(\overline\phi)$, and $f(p):=\bigvee_{w\in \lb p\rb}\sigma_w$.

Then, for all $0\not=w\in W$ and every subformula $\psi$ of $\phi$,

\begin{enumerate}

\item 
if $w\in\lb\, \overline \psi \,\rb$ then $T +\sigma_w\vdash f^\prec_T(\psi)$

\item
if $w\not\in\lb\, \overline\psi\,\rb$ then $T +\sigma_w\vdash \neg f^\prec_T(\psi)$

\end{enumerate}
\end{lemma}

\proof
By an easy induction on the complexity of $\psi$.
\endproof

In the remainder of this section, we shall mainly see how to produce snapshots of a given Kripke model $\mathfrak W$ in a theory $T$. We define the corresponding sentences $\sigma_w$ for $w\in W$ in a standard way as limit statements of certain computable Solovay functions. 

One important notion in all known Solovay-style proofs, including our own, is the notion of a ``code for a $\xi$-derivation of $\phi$.''

\begin{definition}\label{DefProvcode}
Let $X$ be a set of natural numbers. A \emph{code of a $0$-proof of $\phi$ over $X$} is any natural number $m$ satisfying $\provfor_T( m,\phi)$. We assume that every derivation proves a unique formula in our coding, and that this fact is derivable; we also assume that every derivable formula has arbitrarily large derivations\footnote{Usually, derivations are represented as sequences of formulas, and any given derivation can be considered to \emph{only} prove its last formula. Similarly, if derivations are represented as trees, then only the root is considered to be proven. Moreover, in standard proof systems, given a derivation $d$, there are ways to produce a longer derivation with the same end-formula; for example, one may add many redundant copies of an axiom at the beginning of $d$.}.

For $\xi\succ 0$, a triple $\langle \zeta,n,m\rangle$ \emph{codes a $\xi$-derivation of $\phi$ over $X$} if $\xi\succ\zeta$ and

\begin{enumerate}
\item
$n=\ulcorner\psi\urcorner$ for some $\psi$ such that, given $k<\omega$, $\provx \zeta X\psi(\overline k)$ and

\item 
$m$ is a code of a $0$-proof of $\forall x\psi\to\phi$.

\end{enumerate}

Let $\provcodeX(x, \xi, \phi)$ be a formula stating that $x=\langle \zeta,\psi,d\rangle$ codes a $\xi$-derivation of $\phi$ over $X$.
\end{definition}

Of course in $\provcodeX$ the intention is for $X$ to be an iterated provability operator, and we may define
\[\provcode=\forall X(\Provfor(X)\rightarrow \provcodeX).\]
Note that formula $\provcode$ is of rather high complexity ($\Pi^1_1$) which is moreover independent of $\xi$. However, the behavior of $\provcode(x, \xi, \phi)$ is simple in the eyes of $\lambda$ provability whenever $\lambda \seq \xi$, as is expressed in the next lemma.

\begin{lemma}\label{LemmProof}
Let $T$ be a theory extending \base. For $\xi \peq \lambda$ we have
\begin{enumerate}

\item
$\overline T^\prec \ \vdash \ \forall x\forall \xi\forall \phi \  \provcode(x, \xi, \phi) \to \prov \lambda \provcode(\dot x, \dot\xi, \dot\phi)$,\label{LemmProof1}

\item
$\overline T^\prec \ \vdash \ \forall x\forall \xi\forall \phi \ \neg \provcode(x, \xi, \phi) \to \prov \lambda \neg \provcode(\dot x, \dot \xi, \dot \phi)$.\label{LemmProof2}
\end{enumerate}
\end{lemma}

\begin{proof}
We prove the first item. We reason in $\overline T^\prec$ and assume $\provcode(x, \xi, \phi)$. If $\xi=0$, $\provfor_T(d,\phi)$ is $\hat\Delta^0_1$ so $\provx 0X \provfor_T(\overline d,\overline \phi)$. So we assume $\xi \succ 0$ whence $\provfor_T(d, \forall x \psi(x) \to \phi)$ is equivalent to
\begin{equation}\label{equation:definitionOfBeingAProof}
x= \la \psi, \mu, d\ra \wedge \mu {\prec} \xi
 \wedge \forall n \prov \mu \psi (\overline n) \wedge \provfor_T(d, \forall x \psi(x) \to \phi) .
 \end{equation}
All conjuncts other than $\forall n \prov \mu \psi (\overline n)$ are of complexity $\hat\Delta^0_1$ so they --as their negations-- all are 0-provable whence certainly $\lambda$-provable. By Lemma \ref{trans} we obtain $\forall n \prov \mu \psi (\dot n) \to \forall n \prov \mu \prov {\dot\mu} \psi (\dot n)$, and since $\mu \prec \lambda$ we use one application of the $\omega$-rule to see that
\[
\forall n \prov \mu \psi (\dot n) \to \prov \lambda \forall n \prov {\dot\mu} \psi (\dot n).
\]
Since the $\prov \lambda$ predicate is closed under conjunction we have the entire conjunction \eqref{equation:definitionOfBeingAProof} under the scope of the $\prov \lambda$ predicate which was to be shown.

The second item goes analogously now using Lemma \ref{hardaxiom} instead of \ref{trans}. 
\end{proof}

\subsection{Solovay sequences}

Let us define a {\em Solovay sequence} or {\em path;} these sequences are given by a recursion based on provability operators which depends on a parameter $\phi$. Later we will choose an appropriate value of $\phi$ via a fixpoint construction. We shall use the following notation: ${\tt Seq}(x)$ is a $\hat\Delta^0_0$ formula stating that $x$ codes a sequence, ${\tt last}(x)$ is a term that picks out the last element of $x$, $x\sqsubseteq y$ is a $\hat\Delta^0_0$ formula that states that $x$ is an initial segment of $y$, $|x|$ gives the length of $x$ and $x_y$ a term which picks the $y$-coordinate of $x$. As in previous sections, it is not necessary to have these terms available in our language, as we can define their graphs and replace them by pseudo-terms, but we shall write them as such for simplicity of exposition.

We will also define a (pseudo) term ${\tt Lim}$ which gives a formula stating that the paths satsifying $\phi$ ``converge'' to $w$:

\begin{definition}
Define
\[
{\tt Lim}({\phi}, w) \ := \ {\exists s\Big(\phi(\dot s)\wedge \forall s'\sqsupseteq s\ \ \phi(\dot s')\rightarrow{\tt last}(s')=\bar w\Big)}.
\]
\end{definition}

We shall use these formulas to define our recursive paths.

\begin{definition}
Let $\mathfrak W=\langle W,\langle<_n\rangle_{n<N},\lb\cdot\rb\rangle$ be a Kripke frame and let $\bm \lambda=\langle\lambda_n\rangle_{n<N}$ a finite sequence. We define a formula $\PreSolX s\phi$ by
\begin{align*}
\PreSolX s\phi& :=\\
{\tt Seq}& (s) \wedge   {\tt last}(s)\not=0 \\
\wedge&\forall \, x{<}|s|-1  \displaystyle\bigwedge_{w\in W} \Bigg(s_x = \bar w \rightarrow \\
&\Big(\displaystyle{\bigwedge_{n<N}} \displaystyle{\bigwedge_{ v<_n  w}} \neg \provcodeX(x, \lambda_n, \neg{\tt Lim}(\phi, v)\Big) \to s_{x+1} =\bar w  \\
  &\wedge\Big(\displaystyle{\bigvee_{n<N}} \displaystyle{\bigvee_{v<_n w}} \provcodeX(x, \lambda_n, \neg{\tt Lim}(\phi, v)\Big) \to s_{x+1} =\bar v  \Bigg).\\
\end{align*}
We then set $\PreSol s\phi=\forall X\left(\Provfor_T(X)\rightarrow\PreSolX s\phi\right)$.
\end{definition}
We should remark that $s,X,\phi$ are variables and ${\bm \lambda},\mathfrak W$ are external parameters so that we are in fact defining a family of formulas. With this we can say what it means to be a Solovay path.
\begin{definition}[Solovay path]
We define a {\em Solovay path} to be any natural number $s$ satisfying the formula $\Sol {s}$ defined using the fixpoint theorem on the parameter $\phi$ in $\PreSol s\phi$, so that
\[
\base\vdash\Sol {s} \ \leftrightarrow \ \left( \PreSol {{\dot s}}{ \ulcorner{\Sol x}\urcorner}\right).
\]
Further, we say $w$ is a {\em Solovay value} at $i$ if
\[\SolFun wi \  := \ \exists s\,\left(\LongSol s{i}\wedge s_i=w\right)\]
holds, and $w$ is a {\em limit Solovay value} if it satisfies
\[\limfor{w} \ \ := \ \ {\tt Lim}\left(\ulcorner {\Sol x}\urcorner,w\right).\]
\end{definition}

The following shows that Solovay values in fact define a function.

\begin{lemma}\label{solunique}
If $U$ extends \base, $T$ is any representable theory, $\mathfrak W$ is a $\sf J$-frame and $\bm \lambda$ a $\prec$-increasing sequence, it is derivable in $U$ that

\begin{enumerate}
\item $\forall s\forall s'(\Sol s\wedge\Sol {s'}\rightarrow s\sqsubseteq s'\vee s'\sqsubseteq s)$\label{solunique2}
\item $\forall I\exists s \ \ \LongSol sI$ \label{solunique3}
\item $\forall i\exists ! w \ \ \SolFun wi$ and
\item $\forall s\forall j<i \ \ \SolFun  {\bar v}j\wedge\SolFun {\bar w}i\rightarrow \bar w\lleq_0 \bar v$.\label{solunique1}
\end{enumerate}
\end{lemma}

\proof\

\paragraph{1} Clearly it suffices to prove that
\[\Big((\PreSolX s\phi)\wedge(\PreSolX {s'}\phi)\wedge i<|s|\wedge i<|s'|\Big)\rightarrow s_i=s'_i,\]
for then if ${s},{s}'$ are any two paths and, say, $|{\bm s}|\leq |{s}'|$, it follows that $s_i=s'_i$ for all $i<|{s}|$ and thus $s\sqsubseteq {s}'$. Moreover, this formula is arithmetic and hence we may proceed by induction on $i$.

The base case is trivial since $s_0=s'_0=0$. For the inductive step, we assume $w=s_i=s'_i$. Then, we must have that either $\provcodeX(i, \lambda_n, \neg{\tt Lim},\phi,\bar v))$ holds for some $v,n$, or it does not. If it does, then the value of $v$ is uniquely determined (as $i$ may be the code of a derivation of only one formula) and thus $s_{i+1}=s'_{i+1}=v$. Otherwise, ${\bigwedge_n} {\bigwedge_{v<_nw}} \neg \provcodeX(x, \lambda_n, \neg{\tt Lim}(\phi,\bar v))$ holds and $s_{i+1}=s'_{i+1}=w$. Once again the claim follows by introducing universal quantifiers over $X$ and $\phi$.

\paragraph{2} The proof follows the above structure; here we observe that if $s$ is a Solovay path, we may always add one additional element to $s$ depending on which condition is met.

\paragraph{3} This is immediate from items \ref{solunique2} and \ref{solunique3}.

\paragraph{4}

\CurrentProof{
By the recursive definition of a Solovay path, it is always the case that $s_{j+1}\lleq_0 s_j$. Since $\lleq_0$ is transitive, this implies inductively that $s_i\lleq_0 s_j$ whenever $j<i$, and this induction can be easily formalized in $U$.
}

\AlternativeProof{Reasoning within $U$, it will be sufficient to prove the more general claim that
\[\big(j\leq i\wedge\LongSolX siX\phi\big)\rightarrow s_i\lleq_0s_j.\]
This formula is arithmetic and thus we may proceed by induction on $i$. The base case, where $i=j$, follows by reflexivity of $\lleq_0$. For $i+1$, we have by induction hypothesis that $s_i\lleq_0 s_j$. Then, either  $\provcodeX(i, \lambda_n, \neg{\tt Lim}(\phi,v))$ holds for some $v,n$, or it does not. If it does, then $s_{i+1}=v<_n s_i$; otherwise, ${\bigwedge_n} {\bigwedge_{v<_nw}} \neg \provcodeX(x, \lambda_n, \neg{\tt Lim}(\phi,v))$ holds and $s_{i+1}=s_{i}$. In either case we have that $s_{i+1}\lleq_0 s_i\lleq_0 s_j$.

Quantifying over all $\phi$ and all $X$, the claim follows.
}
\endproof

\begin{lemma}\label{important}
Let $\mathfrak W$ be a finite $\sf J$-frame and $\bm \lambda$ a $\prec$-increasing sequence. Suppose further that $U$ extends \base, $\prec$ is recursive and $T$ is representable and $\prec$-introspective, $w\in W$ and $n\leq N$. Then,
\begin{align*}
U^\prec&\vdash
\SolFun {\bar w}k\\
&\rightarrow\prov {\lambda_n}\bigwedge_{v\in W}\, \SolFun {\bar v}{\dot k}\rightarrow \bar v\approx_{n+1} \bar{w}.
\end{align*}
Further,
\begin{align*}
U^\prec\vdash m\leq n\wedge \bar u>_m \bar w \ \ &\\
\wedge\,\SolFun {\bar u}k&\wedge \SolFun {\bar w}{k+1}\\ 
& \rightarrow
\prov {\lambda_n}{} \bigwedge_{v\in W}\SolFun {\bar w}{\dot k}.
\end{align*}
\end{lemma}

\proof
Reasoning within $U$, we will prove both claims simultaneously. To be precise, we show by induction on $k$ that
\begin{align*}
\Provfor&(X)\wedge\LongSolX skX{\ulcorner \Sol x\urcorner}\rightarrow\\ &\bigwedge_{w\in W} \, s_k=w\rightarrow 
\provx {\lambda_n}X\forall x \ \ \LongSol xk\rightarrow x_k\approx_{n+1} \overline{w}\\
&\wedge \bigwedge_{m<n}s_k>_m w\rightarrow
\provx {\lambda_n}{X} \forall x \ \ \LongSol xi\rightarrow
 x_{i+1}=\overline{w}.
\end{align*}

\paragraph{Case 1.} Suppose that $w<_m s_k$ for some $m\leq n$. Then, $k$ codes a $\lambda_m$-derivation of $\nlimfor{\overline w}$, and by Lemma \ref{LemmProof}.\ref{LemmProof1}, 
\begin{align*}
\provx{\lambda_n}{X}\provcode&\left(\bar k, \bar\lambda_m,\nlimfor{\overline w}\right).
\end{align*}
Meanwhile, for $v= s_k$, by our induction hypothesis
\[\provx{\lambda_n}X(\LongSol xk\rightarrow x_k\approx_{n+1}\overline v),\]
but $x_k\approx_{n+1} v<_m w$ provably implies that $x_k<_m w$ by the $\sf J$-frame conditions and thus
\[\provx {\lambda_n}{X}\forall x\Big(\LongSol x{\bar k+1}\rightarrow x_{\bar k+1}= \overline{w}\Big).\]
The claim follows by quantifying over all $X$.

\paragraph{Case 2.} Suppose that for no $m\leq n$ do we have that $s_{k+1}<_m s_{k}$; then as in Case 1 we have that $w<_m x_k$ if and only if $w<_m s_k$ and by Lemma \ref{LemmProof}.\ref{LemmProof2} we have for each such $w$ and $m$ that
\[\prov{\lambda_n}{\neg \provcode\left(\bar k,\bar \lambda_m,\nlimfor {\bar w}\right)}.\]
Thus in view of Lemma \ref{solunique}.\ref{solunique1} and the previous case we must have that
\[x_{k+1}\lleq_{n+1} x_{k}\stackrel{\rm IH}\approx_{n+1}s_{k}\ggeq_{n+1}s_{k+1},\]
which by definition implies that $x_{k+1}\approx_{n+1}s_{k+1}$. Formalizing this reasoning within $T$, it follows that
\[\provx {\lambda_n}{X}\forall x\Big(\LongSol {x}{\bar k+1}\rightarrow x_{\bar k+1}\approx_{n+1} {\bar w}\Big),\]
and once again we conclude the original claim by quantifying over $X$.
\endproof

From here on, it remains to show that the formulas $\limfor{\bar w}$ give a snapshot of our Kripke model.

\begin{lemma}\label{consder}
If $U$ extends \base, $\prec$ is recursive and $T$ is representable and $\prec$-introspective, $m<n < N$ and $v<_m w$ then
\[U \ \ + \ \ \limfor {\bar w} \ \ \vdash \ \  \cons {\lambda_n}\ \ {\limfor{\overline v}}.\]
\end{lemma}

\proof
Towards a contradiction, suppose that 
\[\prov{\lambda_n} \ \ \nlimfor{\overline v}.
\]

It follows that $\prov{\lambda_n} \ \ \nlimfor{\overline v}$, and hence there exists some $i$ which satisfies $\provcode \left(i,\lambda_n,\nlimfor{\overline v}\right)$. Now, by assumption $\limfor w$ holds, and hence we may choose $s$ such that ${\tt Last}(s')=w$ for all $s'\sqsupseteq s$. Recall that we assumed that every derivable formula has arbitrarily large derivations (see Definition \ref{DefProvcode}), and thus we may pick $i>|s|$ such that
\[\provcode \left(i,\lambda_n,\nlimfor{\overline v}\right)\]
holds and, in view of Lemma \ref{solunique}.\ref{solunique3}, a Solovay path $s'$ with $|s'|>i+1$. Then, $s'_i=s'_{i+1}=w$, but by the Solovay recursion we should have $s'_{i+1}=v$, a contradiction.
\endproof

\begin{lemma}\label{related}
If $U,T$ extend \base, $\prec$ is recursive and $T$ is representable and $\prec$-introspective, $w\not=0$ and $n\leq N$ then
\[U^\prec \  + \  \limfor {\bar w} \ \ \vdash \ \  \prov{\lambda_n}{\bigvee_{v\lll_n w}\limfor {\bar v}}.\]
\end{lemma}

\proof
We reason in $U^\prec \  + \  \limfor {\bar w}$. Let $s$ be a Solovay path with with $w={\tt Last}(s)$. Let $k_\ast<|s|$ be the greatest value such that $s_{k_\ast}<_n s_{k_\ast-1}$ if there is such a value; otherwise set $k_\ast=0$.

By Lemma \ref{important},
\begin{equation}\label{EqProvIAst}
\prov{\lambda_n} \Big( \forall x \ \LongSol x{\overline k_\ast} \ \rightarrow \  x_{\bar k_\ast}=\bar s_{k_\ast}\Big).
\end{equation}
Moreover, by Lemma \ref{consder} we have that ${\cons {\lambda_m} \ \limfor {\bar v}}$ so that by Lemma \ref{hardaxiom} we also have
\[\prov {\lambda_n}\bigwedge_{m<n}\bigwedge_{v<_m w}{\cons {\bar \lambda_m} \ \limfor {\bar v}},\]
from which it follows using Lemma \ref{solunique}.\ref{solunique1} that
\begin{equation}\label{EqTIAst}
\prov{\lambda_n}  \ \forall x \ \ \bar k_\ast<j\wedge \LongSol xj\rightarrow x_j\lleq_n x_{\bar k_\ast}.
\end{equation}

Putting \eqref{EqProvIAst} and \eqref{EqTIAst} together, along with the fact that $s_{k_\ast}\ggeq_{n+1}w$ we see that
\[\prov{\lambda_n} \ \forall x \ \ \LongSol x{\bar k_\ast}\rightarrow {\tt Last}(x)\llleq_n \overline w.\]

It remains to disprove the case that ${\tt Last}(x)=\overline w$. For this, choose the least value of $k$ such that $s_{k+1}=w$; note that this value is well-defined since $s_0=0$. It then follows that $\provcode (k,\lambda_n,\nlimfor{\overline w})$, and in view of Lemma \ref{LemmProof}.\ref{LemmProof1} the latter clearly implies that $\prov{\lambda_n}\nlimfor{\bar w}$, as required.
\endproof

We are now ready to prove that $\limfor w$ provide a snapshot of our $\sf J$-model.

\begin{lemma}\label{nearfinal}
Let $T$ be any sound, $\prec$-introspective theory extending \base. Given a finite $\sf J$-frame $\mathfrak W$ with root $0$ and any $\prec$-increasing sequence $\bm \lambda$ set
\[\limvec=\left\langle\limfor {\bar w}:w\in W\right\rangle.\]
Then,
\[\snapshot{\bm \lambda}{\limvec} .\]
\end{lemma}

\proof We must check each of the conditions of Definition \ref{defrepre}.

\paragraph{1} For the first, suppose that $\limfor w$ and $\limfor v$ hold. In view of $\limfor w$, pick a Solovay path $s$ such that any extension of $s$ has last element $w$ and similarly $s'$ such that any extension of $s'$ has last element $v$. By Lemma \ref{solunique}.\ref{solunique2}, either $s\sqsupseteq s'$ or $s'\sqsupseteq s$; in either case, it follows by $\limfor w\wedge\limfor  v$ that $w=v$.

\paragraph{2-3} The second condition is Lemma \ref{consder} and the third, Lemma \ref{related}.

\paragraph{4} For the fourth, we must use the fact that $\prov\lambda$ is sound for all $\lambda$ and proceed by induction on $\ll_0$ to show that, if $s$ is a Solovay path and $w\not=0$, then $s_i\not = 0$ for all $i$. For indeed, if $s_i=w\not=0$ for some $i$, picking the minimal such $i$ we see that $i$ codes a $\lambda_n$-derivation of $\nlimfor {\bar w}$; hence, by soundness we must have that, for some $s'\sqsupseteq s$, ${\tt last}(s')\not={w}$. But then, by Lemma \ref{solunique}.\ref{solunique1}, ${\tt last}(s')=v\ll_0 w$, but by induction on $v\ll_0 w$ there can be no such path.

We conclude that any Solovay path is identically zero, so that the formula $\limfor {\bar 0}$ is true.
\endproof

We may finally prove our main completeness result.

\proof[Proof of Theorem \ref{complete}] We have already seen that the logic is sound.

For the other, if $\phi$ is consistent over $\glp_\prec$, then by Lemma \ref{glptoj}, $M^+(\phi)\wedge\phi$ is consistent over $\sf J$ and thus by Lemma \ref{jcomp}, $M^+(\bar\phi)\wedge\bar\phi$ can be satisfied on a world $w_\ast$ of some stratified $\sf J$-model $\mathfrak W'$. Define $\mathfrak W$ by adding $0$ as a root to $\mathfrak W'$ and let $\bm\lambda$ be the modalities appearing in $\phi$. By Lemma \ref{nearfinal},
\[\snapshot{\bm \lambda}{\limvec}\]
so that by Lemma \ref{truthlemma}.1, \[T^\prec \ \ + \ \ \limfor{w_\ast} \ \ \vdash \ \  f^\prec_T(\phi).\] Hence, by lemma \ref{truthlemma}.2, $\mathbb N\models \ps_T{f^\prec_T(\phi)}$, i.e. $f^\prec_T(\phi)$ is consistent over $T$.
\endproof

\appendix

\section{Alternative provability predicates}\label{section:AlternativePresentations}

In this section we shall briefly discuss some variants of our provability predicates. We do so in an informal setting and in particular shall refer to defining recursions rather than formalizations in second order logic. Moreover, we shall on occasion not be too concerned about the amount of transfinite induction needed in the arguments.

We could consider an apparently slightly weaker notion of $\alpha$ provability -- let us write $[\alpha]^{\prec,w}_T$ -- defined by the following recursion:
\[
[\alpha]^{\prec,w}_T \phi\ :\Leftrightarrow \ \Box_T\phi \vee \exists \psi \, \exists \, \beta{\prec}\alpha (\forall n\, [\beta]^{\prec,w}_T \psi (\overline n) \wedge [\beta]^{\prec,w}_T (\forall x\, \psi(x) \to \phi)).
\]
However, it is easy to see by transfinite induction that  $[\alpha]^{\prec,w}_T \phi \ \Leftrightarrow \ [\alpha]^\prec_T \phi$. The $\Leftarrow$ direction is obvious. For the other direction we assume that we can formalize the notion of $[\alpha]^{\prec,w}_T$ just like $\prov \alpha {}$ and prove all the necessary lemmata like monotonicity, distribution axioms, etc. Suppose that $\forall n \, [\beta]^{\prec,w}_T \psi (\overline n) \wedge [\beta]^{\prec,w}_T (\forall x\, \psi(x) \to \phi)$
for some formula $\psi$ and ordinal $\beta\prec \alpha$.
Then, clearly also $\forall n \, [\beta]^{\prec,w}_T \Big(\psi (\overline n) \wedge (\forall x\, \psi(x) \to \phi)\Big).$
But, as $\prov 0 \Big((\forall x \, \psi (x) \wedge (\forall x\, \psi(x) \to \phi))\to \phi
\Big)$ we get 
\[
\forall n \, [\beta]^{\prec,w}_T \Big(\psi (\overline n) \wedge (\forall x\, \psi(x) \to \phi)\Big) \wedge [0]\Big((\forall x \, \psi (x) \wedge (\forall x\, \psi(x) \to \phi))\to \phi
\Big)
\]
which by the induction hypothesis for $\beta$ is just $\prov\alpha\phi$.\\
\medskip

Note that in our definition of $\prov \alpha {}$ there is still some uniformity present in that we choose one particular $\beta \prec \alpha$ with $\prov \beta\psi (\overline n)$ for all numbers $n$. We can make this $\beta$ also dependent on $n$. In a sense, this boils down to diagonalizing at limit ordinals. Thus, we define our notion of $[\alpha]^{\prec,d}_T$ as follows:
\[
[\alpha]^{\prec,d}_T \phi \ \ :\Leftrightarrow \ \Box_T \phi \ \vee \ \exists \, \psi\, \Big( \forall \, n\, \exists \beta_n {\prec} \alpha \ \big( [\beta_n]^{\prec,d}_T\psi(\overline n) \ \wedge \ \Box (\forall x \, \psi(x) \to \phi) \big)\Big).
\]
By an argument similar as before, we see that, also in this notion we can replace the $\Box (\forall x \, \psi(x) \to \phi)$ by $\exists \, \gamma {\prec} \alpha \ [\gamma]^{\prec,d}_T (\forall x \, \psi(x) \to \phi)$ without losing any strength.
This new notion of provability is related to $\prov \alpha {}$ in a simple fashion as is expressed in Lemma \ref{theorem:diagonalLimitsAndSuccessors} below. Again, we assume that we can formalize the notion $[\alpha]^{\prec,d}_T$ in a suitable way so that the basic properties are provable. We first state a simple but useful observation.

\begin{lemma}\label{theorem:sucessorProvability}\ 
\begin{enumerate}
\item
$\prov {\alpha +1} \phi \ \Leftrightarrow \ \exists \psi \Big( \forall n\ \prov \alpha \psi(\overline n) \ \wedge \ \Box (\forall x \psi(x) \to \phi) \Big)$
\item
$[\alpha+1]^{\prec,d}_T \phi \ \Leftrightarrow \ \exists \psi \Big( \forall n\ [\alpha]^{\prec,d}_T\psi(\overline n) \ \wedge \ \Box (\forall x \psi(x) \to \phi) \Big)$
\end{enumerate}
\end{lemma}

\begin{proof}
This follows directly from the definition and monotonicity.
\end{proof}

\begin{lemma}\label{theorem:diagonalLimitsAndSuccessors}\ 
\begin{enumerate}
\item
$\prov n\phi \ \Leftrightarrow [n]^{\prec,d}_T\phi$ for $n \in \omega$; 
	
\item
$\prov  {\alpha +1}\phi \ \Leftrightarrow [\alpha]^{\prec,d}_T\phi$ for $\alpha \seq \omega$.

\end{enumerate}
\end{lemma}

\CurrentProof{
\proof
The proofs proceed by induction on $n$ and $\alpha$, respectively, and we omit them.\endproof
}

As can be seen, there is a fair amount of freedom in defining transfinite iterations of the $\omega$-rule. We chose the current paper's presentation both for the sake of simplicity and because a more refined hierarchy is in general terms more convenient; after all, it is easy to remove intermediate operators later if they are not needed. We also suspect it will be the appropriate notion useful later for a $\Pi^0_1$-ordinal analysis of second-order arithmetics, a goal which now seems well within our reach.

\section{An afterword on the choice of our base theory}\label{AppBaseTheory}

In this paper we have shown sound and completeness of the logic $\glp_\prec$ for the interpretation where each $[\xi]$ modality is interpreted as ``provable in \base using at most $\xi$ nested applications of the omega-rule". The main applications we have in mind with this result is to provide $\Pi^0_1$ ordinal analyses of theories much stronger than \pa in the style of Beklemishev (\cite{Beklemishev:2004:ProvabilityAlgebrasAndOrdinals}). For the mere soundness of the logic however, there were quite some strong principles needed: the existence of an iterated provability class, plus a certain amount of transfinite induction. We shall discuss here how these principles fit into the intended application of ordinal analyses.

%The $\Pi^0_1$ ordinal of a theory $U$ --denoted by $|U|_{\Pi^0_1}$-- is defined as ``how often iterating adding conistency" is needed when starting with some finitistic base theory, say Elementary Arithmetic (\ea), so as to end up with a largest possible part of the $\Pi_1^0$ consequences of $U$. More in detail: for ordinals $\alpha, \beta$ in our notation system, let $(\ea)_0 := \ea$, $(\ea)_{\alpha +1} := (\ea)_\alpha + {\sf Con}((\ea)_\alpha)$ and $(\ea)_{\alpha}: = \cup_{\beta\prec \alpha}(\ea)_{\beta}$ for limit ordinals $\alpha$. Then $|U|_{\Pi^0_1}:= \sup_\alpha (\ea)_\alpha \subseteq U$.
%If one works with a natural ordinal notation system for large enough ordinals then under some fairly general conditions one can show that the smallest ordinal $\alpha$ so that $\ea + \transin (\Pi^0_1, \alpha) \vdash {\sf Con}(U)$. Results as these are seen as partial realizations of Hilbert's program in the sense that over finitsitic mathematics one can prove the consistency of a strong theory with just one additional non-finitist ingredient. As such one could say that $U$ is safeguarded by this method. If one accepts this method of safeguarding it makes philosophically speaking sense that one may use $U$ itself as base theory to safeguard even stronger theories by the same method. It is in this perspective that having $\base$ as our base theory is not a bad thing since this theory has already be safeguarded by the method  

A consistency proof of $U$ in Elementary Arithmetic (\ea) plus $\transin (\hat \Pi^0_1, \prec)$ (using a natural ordinal notation system for large enough ordinals) is closely related to the $\Pi^0_1$ ordinal analysis of $U$ and we shall focus our discussion on such a consistency proof. Such a consistency proof can be seen as a partial realization of Hilbert's program in the sense that over finitsitic mathematics one can prove the consistency of a strong theory with just one additional non-finitist ingredient. As such one could say that $U$ is safeguarded by this method. If one accepts this method of safeguarding it makes philosophically speaking sense that one may use $U$ itself as new base theory to safeguard even stronger theories by the same method. It is in this perspective that having $\base$ as our base theory is not a bad thing since \base has already been safeguarded over \ea using some amount of transfinite induction. However, for technical reasons it might be desirable to dispense with such an intermediate step.

In \cite{Joosten:2013:RelativeOrdinalAnalysis} it is noted that soundness of $\glp_{\prec}$ suffices to perform a consistency proof and completeness is actually not needed. But also in the soundness proof presented in this paper we needed transfinite induction as well as resorting to the introspective closure of a theory. However, as we have seen in Corollary \ref{theorem:eachTheoryIsEquiConsistentToIntrospective}, for the sake of consistency-strength it is irrelevant to consider either a theory or its introspective closure as both theories are provably equiconsistent. Thus, to conclude, let us consider the amount of transfinite induction needed in our soundness proof of $\glp_\prec$. 

First of all, let us note that our soundness proof of $\glp_\prec$ uses at most $\transin (\Pi^0_1, \prec)$. In a sense, this is not bad at all, because it is exactly this ingredient (in parameter free form) that is added \ea to perform a consistency proof of our target theory. So, by adding this amount of transfinite induction (with parameters) to \ea, we get access to exactly the soundness of $\glp_\prec$ needed to perform this consistency proof, were it not for the case that our base theory was taken to be \base, not \ea. This choice of \base has been mainly to simplify our exposition and the needed amount of arithmetic can be pushed down a lot further. 

We observed that to prove $\transin (\Pi^0_1, \prec)$ over $\ea^{\prec}$ we need $\Sigma^0_1$ comprehension. But clearly $\Sigma^0_1$ comprehension with second order parameters proves \base. However, close inspection of the proofs in our paper shows that the only free set-parameters needed are occurrences of iterated provability classes. Thus if we enrich our language with a constant $\pi$ with an axiom stating that $\pi$ is an iterated provability class  (see Lemma \ref{theorem:ExtendingLanguageNeedsOnlyParameterFreeComprehensionToShowIntrospectivity}), we can do with parameter free comprehension since then $\transin (\hat \Pi^0_1, \prec)$ suffices. 

However, we can do better still in the sense that we need less comprehension by allowing slightly stronger well-ordering assumptions as we shall see in the next lemma. 

Let us fix some bijective coding of $\alpha$ on the naturals, and let $\prec$ be a primitive recursive well-order on $\alpha$. Let $<$ denote the usual ordering on the natural numbers. Using a bijective pairing function we define a new relation 
\[
\la \xi, n \ra \prec' \la \zeta, m \ra:= \xi \prec \zeta \vee (\xi = \zeta \wedge n < m).
\]
Clearly, $\prec'$ provably defines a relation of order type $\omega \cdot \alpha$. With this notation we can now state our lemma.

\begin{lemma}
$\hat \Delta_0^0\mbox{-}\compax + {\tt wo} (\omega \cdot \alpha) \vdash \transin (\hat \Pi^0_1, \alpha)$.
\end{lemma}

\begin{proof}
By the usual argument we see that $\hat \Delta_0^0\mbox{-}\compax + {\tt wo} (\omega \cdot \alpha) \vdash \transin (\hat \Delta^0_0, \omega \cdot \alpha)$. Thus, we shall proof $\transin (\hat \Pi^0_1, \alpha)$ using $\transin (\hat \Delta^0_0, \omega \cdot \alpha)$. Let $\varphi(z,x)$ be some $\hat \Delta^0_0$ formula and assume 
\begin{equation}\label{equation:ShallICompareTheeToASummersDayPuesNoPorqueEresUnPutoPendejo}
\forall x \ (\forall \, y {\prec} x \ \forall z\varphi (z,y) \to \forall z \varphi (z,x)).
\end{equation}
If we assume that $\forall \,  y{\prec'} x \ \varphi(y_1, y_0)$, using \eqref{equation:ShallICompareTheeToASummersDayPuesNoPorqueEresUnPutoPendejo} we get $\varphi(x_1,x_0)$ whence by $\transin (\hat \Delta^0_0, \omega \cdot \alpha)$ we obtain $\forall x \forall z \varphi (z,x)$.
\end{proof}

Note that $\omega \cdot \alpha$ is not much larger than $\alpha$. In particular, if the last term in CNF of $\alpha$ is at least $\omega^\omega$ we get that $\omega\cdot \alpha = \alpha$. Thus, for natural proof theoretical ordinals we have this equation whence we get the extra induction for free.

\bibliographystyle{plain}
\bibliography{References}
\end{document}